\documentclass[11pt]{amsart}
\usepackage[english]{babel}
\usepackage{enumerate}
\usepackage{amsmath,amssymb}

\ifx\pdfpageheight\undefined
   \usepackage[dvips,colorlinks=true,linkcolor=blue,citecolor=red,%
      urlcolor=green]{hyperref}
   \usepackage[dvips]{graphicx}
   \makeatletter
   \edef\Gin@extensions{\Gin@extensions,.mps}
   \makeatother
\else
 \usepackage[pdftex]{graphicx}
   \usepackage[bookmarksopen=false,pdftex=true,breaklinks=true,%
      backref=page,pagebackref=true,plainpages=false,%
      hyperindex=true,pdfstartview=FitH,colorlinks=true,%
      pdfpagelabels=true,colorlinks=true,linkcolor=blue,%
      citecolor=red,urlcolor=green,hypertexnames=false%
      ]%
   {hyperref}
\fi
\usepackage{bm}
\usepackage{amsmath}
\usepackage{tabularx}
\usepackage{color, colortbl}
\usepackage{url}

\makeindex
\catcode`\<=\active \def<{
\fontencoding{T1}\selectfont\symbol{60}\fontencoding{\encodingdefault}}
\catcode`\>=\active \def>{
\fontencoding{T1}\selectfont\symbol{62}\fontencoding{\encodingdefault}}
\catcode`\|=\active \def|{
\fontencoding{T1}\selectfont\symbol{124}\fontencoding{\encodingdefault}}

\newcommand{\noplus}{}

\newcommand{\tmop}[1]{\ensuremath{\operatorname{#1}}}

\newcommand{\tmstrong}[1]{\textbf{#1}}

\definecolor{LightCyan}{rgb}{0.88,1,1}

\newtheorem{theorem}{Theorem}
\newtheorem{lemma}{Lemma}
\newtheorem{corollary}{Corollary}
\newtheorem{proposition}{Proposition}

\theoremstyle{definition}
\newtheorem{definition}{Definition}
\newtheorem{example}{Example}

\newtheorem{notation}{Notation}

\theoremstyle{remark}
\newtheorem{remark}{Remark}

\newtheoremstyle{break}  
  {3pt}   
  {11pt}   
  {\normalfont}  
  {0pt}       
  {\scshape} 
  {}         
  {4pt}  
  {}          
\theoremstyle{break}

\newcommand{\ra}{\rangle}
\newcommand{\la}{\langle}

\newcommand{\F}{\mathbb{F}}
\newcommand{\R}{\mathbf{R}}
\newcommand{\C}{\mathbf{C}}
\newcommand{\Q}{\mathbb{Q}}
\newcommand{\PP}{\mathbb{P}}
\newcommand{\ZZ}{\mathrm{Zer}}
\newcommand{\RR}{\mathrm{Reali}}

\newcommand{\Z}{\mathbb{Z}}
\newcommand{\eps}{\varepsilon}

\DeclareMathOperator{\SIGN}{SIGN}
\DeclareMathOperator{\sign}{sign}
\DeclareMathOperator{\Ext}{Ext}
\DeclareMathOperator{\Def}{Def}

\newcommand{\HH}{\mathrm{H}}
\newcommand{\card}{\mathrm{card}}

\newcommand {\hide}[1]{}

\newcommand{\X}{\mathbf{X}}
\newcommand{\x}{\mathbf{x}}
\newcommand{\Y}{\mathbf{Y}}
\newcommand{\y}{\mathbf{y}}
\newcommand{\ZB}{\mathbf{Z}}
\newcommand{\zb}{\mathbf{z}}
\newcommand{\kk}{\mathbf{k}}
\newcommand{\dd}{\mathbf{d}}
\newcommand{\length}{\mathrm{length}}

\newcommand{\symm}{\mathrm{symm}}
\newcommand{\boldsigma}{\boldsymbol{\sigma}}
\newcommand{\Hank}{\mathrm{Hank}}

\newcommand{\boldpi}{\boldsymbol{\pi}}
\newcommand{\boldPi}{\boldsymbol{\Pi}}
\newcommand{\fixed}{\mathrm{fixed}}
\newcommand{\orbit}{\mathrm{orbit}}
\newcommand{\ind}{\mathrm{ind}}
\newcommand{\grad}{\mathrm{grad}}
\newcommand{\Hess}{\mathrm{Hess}}
\newcommand{\spanof}{\mathrm{span}}
\newcommand{\Sym}{\mathrm{Sym}}

\newcommand{\W}{\mathcal{W}}

\begin{document}
\title[Bounding the equivariant Betti numbers]
{Bounding the equivariant Betti numbers  of symmetric semi-algebraic sets}

\author{Saugata Basu}
\address{Department of Mathematics\\
Purdue University, West Lafayette\\
USA}
\email{sbasu@math.purdue.edu
}

\author{Cordian Riener}
\address{Aalto Science Institute\\
Aalto University, Espoo\\
Finland}
\email{cordian.riener@aalto.fi}

\thanks{The first author was partially supported by NSF grants
CCF-1319080 and DMS-1161629.}

\maketitle

\begin{abstract}
Let $\R$  be a real closed field. 
The problem of obtaining tight bounds on the Betti numbers of semi-algebraic subsets of 
$\R^k$ in terms of the number and degrees of the defining polynomials 
has been an important problem in real algebraic geometry with the first results due to
Ole{\u\i}nik and Petrovski{\u\i}, Thom and Milnor. These bounds are all exponential in the
number of variables $k$. Motivated by several applications in
real algebraic geometry, as well as in theoretical computer science, where such bounds have found 
applications, 
we consider in this paper the problem of bounding  
the \emph{equivariant} Betti numbers of 
symmetric algebraic and semi-algebraic subsets of $\R^k$. We obtain several asymptotically
tight upper bounds.  
In particular, we prove that if $S\subset \R^k$ 
is a semi-algebraic subset defined by a finite set of $s$ symmetric polynomials of degree at most $d$, 
then the sum of the $\mathfrak{S}_k$-equivariant Betti numbers of $S$ with coefficients in
$\mathbb{Q}$ is bounded by $(skd)^{O(d)}$.
Unlike the classical bounds on the ordinary
Betti numbers of real algebraic varieties and semi-algebraic sets, the 
above bound is polynomial in $k$ when the degrees of the defining polynomials are bounded by a 
constant.
As an application we improve the best known bound on the ordinary Betti numbers of the projection of a
compact algebraic set improving for any fixed degree the best previously known bound for this problem 
due to Gabrielov, Vorobjov and Zell. 
 \end{abstract}

\tableofcontents
\section{Introduction}
The problem of bounding the Betti numbers of semi-algebraic sets defined over
the real numbers has a long history, and has attracted the attention of many
researchers -- starting from the first results due to 
Ole{\u\i}nik and Petrovski{\u\i}
{\cite{OP}}, followed by Thom {\cite{T}}, Milnor {\cite{Milnor2}}. Aside from
their intrinsic mathematical interest from the point of view of real algebraic
geometry, these bounds have found applications in diverse areas -- most
notably in discrete and computational geometry (see for example
{\cite{BPR10}}), as well as in theoretical computer science
{\cite{Yao94,Pardo96,Bjorner-Lovasz}}. Very recently, studying the probability
distribution of these numbers for randomly chosen real varieties have also
become an important topic of research {\cite{Gayet-Welschinger}}.

In this paper we study the topological complexity of real varieties, as well
as semi-algebraic sets, which have symmetry. We will see  that the
ordinary Betti numbers of symmetric semi-algebraic sets can be
(asymptotically) as large as in the general non-symmetric case. So studying
the growth of Betti numbers of symmetric semi-algebraic sets is not very
interesting on its own. However, for symmetric semi-algebraic sets it is
natural to consider their \emph{equivariant} Betti numbers. The equivariant
Betti numbers (with coefficients in a field of characteristic $0$) equals in
this case the Betti numbers of their orbit spaces -- and here some
interesting structure emerges. For instance, unlike in the non-equivariant
situation the behavior of these equivariant Betti numbers of real and complex
varieties drastically differ from each other. Moreover, in both cases the
higher dimensional equivariant cohomology groups vanish -- and the dimension
of vanishing only depends on the degrees of the polynomials defining the
variety, and is independent of the dimension of the ambient space.
To our knowledge quantitative studies on the topology of symmetric semi-algebraic sets, in particular
obtaining tight bounds on their equivariant Betti numbers, have not been undertaken previously.
We prove asymptotically tight bounds on the equivariant Betti numbers of
symmetric semi-algebraic sets
as well as give an application of our results in a non-equivariant setting.

For the remainder of the paper we fix a real closed field $\R$, and we denote by
$\C$  the algebraic closure of $\R$. 

{\bf Outline of the paper:} The paper is structured as follows. 
In \S \ref{subsec:motivation} we discuss some history and motivation behind studying the problem
of bounding the equivariant Betti numbers of symmetric semi-algebraic sets. In \S \ref{subsec:intro-complexity}
we give a brief introduction  to and overview of known bounds on the Betti numbers of semi-algebraic subsets in $\R^k$ as well as of complex sub-varieties of $\C^k$. In \S \ref{subsec:intro-complexity-symmetric} 
we introduce the basic definitions and certain basic results related to  equivariant (co)homology.
In \S \ref{subsec:intro-comparison}  we highlight some fundamental  differences in the behavior of the equivariant Betti numbers of real as opposed to complex algebraic varieties. 
In \S \ref{sec:main-results} we state the main results of this paper.
We give an outline of the proofs of
the results in \S \ref{subsec:outline}.    

The rest of the paper is devoted to the proofs of these results.
In \S \ref{sec:prelim}, we
recall certain facts from real algebraic geometry and topology that are needed
for the proofs of the main theorems. These include definitions of certain real
closed extensions of the ground field $\R$ consisting of algebraic Puiseux
series with coefficients in $\R$. We also recall some basic inequalities
amongst the Betti numbers which are consequences of the Mayer-Vietoris exact
sequence. In \S \ref{sec:deformation}, we define certain equivariant
deformations of symmetric varieties and prove some topological properties of
these deformations, that mirror similar ones in the non-equivariant case. We prove the 
main theorems in \S \ref{sec:proofs}. 

Finally, we end with some open questions in
\S \ref{sec:conclusion}.

\subsection{Motivation}
\label{subsec:motivation}

There are several different motivations behind studying
the equivariant Betti numbers of symmetric semi-algebraic sets. One motivation
comes from computational complexity theory. It is a well known phenomenon that
the worst case topological complexity of a class of semi-algebraic sets
reflects the \emph{computational hardness} of testing whether a given set
in this class is non-empty, as well as computing topological invariants such
as the Betti numbers of such sets. For instance, it is an
{\textbf{$\tmop{NP}$}}-hard problem (in the Blum-Shub-Smale model) to decide
if a given real algebraic variety $V \subset \R^{k}$ defined by one polynomial
equation of degree at most $4$ is empty or not {\cite{BSS89}}. The Betti
numbers of such varieties can be exponentially large in $k$. In contrast, the
same problem of deciding emptiness, as well as computing other topological
invariants of real varieties defined by a fixed number of quadrics in $\R^{k}$
can be solved with polynomial complexity \cite{Bar93,BP'R07joa}.
(Note that while a real variety defined by any number of at most quadratic equations can obviously be defined by a single polynomial equation of degree $\leq 4$ by taking a sum of squares, not all quartic polynomials
in $k$ variables can be written as a sum of squares of some constant number of quadratic polynomials as
$k\rightarrow \infty$, and thus the last statement does not contradict the previous one.)
The Betti numbers of such sets can also be bounded by a polynomial function of
$k$ \cite{Bar97,BP'R07jems}. This close connection between the
worst case upper bound on the Betti numbers, and the algorithmic complexity of
computing topological invariants, breaks down if one considers the class of
``symmetric'' real varieties. On one hand the topological complexity in terms
of the Betti numbers of such sets can be as big as in the non-symmetric
situation (see Example \ref{ex:zero-dimensional}). On the other hand,
there exist algorithms whose complexity depend polynomially in the number of
variables (for fixed degrees) for testing emptiness of such sets
\cite{Timofte03,Riener}. This dichotomy suggests that perhaps the
topological complexity of symmetric varieties, and semi-algebraic sets is
better reflected by their equivariant Betti numbers rather than the ordinary
ones. The results of the current paper (which show that the equivariant Betti
numbers of real varieties and semi-algebraic sets are polynomially bounded for
fixed degrees) agree with this intuition. We also note that studying the 
computational complexity of symmetric vs. non-symmetric versions of problems in linear algebra
and algebraic geometry
is an active field of research -- see for example \cite{Hillar} for several results
of this kind for computational problems involving high-dimensional tensors. 

Our second motivation is more concrete and leads to an improvement in certain
situations of an important result proved by Gabrielov, Vorobjov and Zell
{\cite{GVZ04}} who proved a bound on the ordinary Betti numbers of the image
under projection of a semi-algebraic set, in terms of the number and degrees
of polynomials defining the original set. The bound is obtained by bounding
the dimensions of certain groups occurring as the $E_{1}$-term of a certain
spectral sequence. It turns out that there is an action of the symmetric group
on this spectral sequence, and quotienting out this action yields a better
approximation to the homology groups of the image than the original spectral
sequence. Our bound on the equivariant Betti numbers can now be used to bound
the dimension of this quotient object. We explain this consequence of our
results in \S \ref{subsec:non-equivariant}.

Before proceeding further we first fix some notation and recall some
classical tight upper bounds on the Betti numbers of general (i.e. not
necessarily symmetric) real (respectively complex)  varieties, in terms of the
degrees of the defining polynomials and the dimension of the ambient space.
Obtaining such bounds has been an important area of research in quantitative real (respectively complex)
algebraic geometry.

\subsection{Topological complexity of complex varieties and real
semi-algebraic sets}
\label{subsec:intro-complexity}
\begin{notation}
  For $P \in \R [X_{1} , \ldots ,X_{k} ]$ (respectively $P \in \C [ X_{1} ,
  \ldots ,X_{k} ]$) we denote by $\ZZ (P, \R^{k})$ (respectively
  $\ZZ (P, \C^{k})$) the set of zeros of $P$ in
  $\R^{k}$(respectively $\C^{k}$). More generally, for any finite set
  $\mathcal{P} \subset \R [ X_{1} , \ldots ,X_{k} ]$ (respectively
  $\mathcal{P} \subset \C [ X_{1} , \ldots ,X_{k} ]$), we denote by $\ZZ
  (\mathcal{P}, \R^{k})$ (respectively $\ZZ(\mathcal{P},
  \C^{k})$) the set of common zeros of $\mathcal{P}$ in
  $\R^{k}$(respectively $\C^{k}$). 
\end{notation}

\begin{notation}
  \label{not:sign-condition} For any finite family of polynomials $\mathcal{P}
  \subset \R [ X_{1} , \ldots ,X_{k} ]$, we call an element $\sigma \in \{
  0,1,-1 \}^{\mathcal{P}}$, a \emph{sign condition} on $\mathcal{P}$. For
  any semi-algebraic set $Z \subset \R^{k}$, and a sign condition $\sigma \in
  \{ 0,1,-1 \}^{\mathcal{P}}$, we denote by $\RR (\sigma ,Z)$ the
  semi-algebraic set defined by 
  \[
  \left\{ \x \in Z \mid \sign (P (\x)) = \sigma (P)  ,P \in \mathcal{P} \right\},
  \]
  and call it the
  \emph{realization} of $\sigma$ on $Z$. 
  More generally, we call any
  Boolean formula $\Phi$ with atoms, 
  $P \;\sim\; 0, P \in \mathcal{P}$ where $\sim$ is one of $=,>,$ or $<$, 
  to be a \emph{$\mathcal{P}$-formula}. We call the realization of $\Phi$,
  namely the semi-algebraic set
  \begin{eqnarray*}
    \RR (\Phi , \R^{k}) & = & \{ \x \in \R^{k} \mid
    \Phi (\x) \}
  \end{eqnarray*}
  a \emph{$\mathcal{P}$-semi-algebraic set}. Finally, we call a Boolean
  formula without negations, and with atoms 
  $P \;\sim\; 0, P\in \mathcal{P}$ where $\sim$ is one of  $\leq,\geq$, 
  to be a \emph{$\mathcal{P}$-closed formula}, and we call
  the realization, 
  $\RR(\Phi , \R^{k})$, a \emph{$\mathcal{P}$-closed semi-algebraic set}.
\end{notation}

\begin{notation}
  For any semi-algebraic set or a complex variety $X$, and a field of
  coefficients $\F$, we will denote by $\HH^{i} (X,\F)$ the
  \emph{$i$-th cohomology group} of $X$ with coefficients in
  $\F$, by $b_{i} (X,\F) =  \dim_{\F}  \HH^{i} (X,\F)$, 
  and by $b(X,\F) = \sum_{i \geq 0}b_i(X,\F)$. Note that defining the cohomology groups of
  semi-algebraic sets over arbitrary (possibly non-archimedean) real closed
  fields requires some care, and we refer the reader to 
  \cite[Chapter 6]{BPRbook2} for details. Roughly speaking, for a closed and bounded
  semi-algebraic set $S$, $\HH^{i} (S,\F)$ is defined as the $i$-th
  simplicial cohomology group associated to a semi-algebraic triangulation of
  $S$.  For a general semi-algebraic set $S$, $\HH^{i} (S,\F)$ is
  defined as the $i$-th cohomology group of a closed and bounded
  semi-algebraic replacement of $S$, which is semi-algebraically homotopy
  equivalent to it. This definition is clearly invariant under semi-algebraic
  homotopy equivalences, and coincides with ordinary singular cohomology
  groups for semi-algebraic sets defined over $\mathbb{R}$.
\end{notation}

The following classical result, which gives an upper bound on the Betti
numbers of a real variety in terms of the degree of the defining polynomial
and the number of variables, is due to 
Ole{\u\i}nik and Petrovski{\u\i} \cite{OP},
Thom \cite{T} and Milnor \cite{Milnor2}.

\begin{theorem}
  \label{thm:betti-bound-algebraic}{\cite{OP,T,Milnor2}} Let
  $Q  \in \R [ X_{1} , \ldots ,X_{k} ]$ be a polynomial with $\deg (Q)  
  \leq  d$. Then, for any field of coefficients $\F,$
  \begin{eqnarray*}
    b(\ZZ (Q, \R^{k}) ,\F) & \leq & d (2d-1)^{k-1} .
  \end{eqnarray*}
\end{theorem}

By separating the real and imaginary parts of complex polynomials and taking
their sums of squares, one obtains as an immediate corollary:

\begin{corollary}
  \label{cor:betti-bound-algebraic-complex}Let $\mathcal{Q}  \subset \C [
  X_{1} , \ldots ,X_{k} ]$ be a finite set of polynomials with $\deg (Q)  
  \leq  d  ,Q \in \mathcal{Q}$. Then, for any field of coefficients
  $\F,$
  \begin{eqnarray*}
    b(\ZZ (\mathcal{Q}, \C^{k})
    ,\F) & \leq & 2d (4d-1)^{2k-1} .
  \end{eqnarray*}
\end{corollary}

In the semi-algebraic case, we have the following bounds. 

\begin{theorem}
  \label{thm:betti-bound-sa}{\cite{Milnor2}} Let $S \subset \R^{k}$ be a 
  basic closed semi-algebraic set
  (i.e. a semi-algebraic set defined by a finite conjunction of weak polynomial inequalities)
   defined by $P_{1} \geq 0, \ldots ,P_{s} \geq 0$,
  and the degree of each $P_{i}$ is bounded by $d$. Then, for any field of
  coefficients $\F,$
  \begin{eqnarray*}
    b(S,\F) & \leq & sd (2sd-1)^{k-1} .
  \end{eqnarray*}
\end{theorem}

\begin{theorem}
  \label{thm:betti-bound-sa-general}{\cite{BPRbook2,GV07}} Let
  $\mathcal{P} \subset \R [ X_{1} , \ldots ,X_{k} ]$ be a finite family of
  polynomials with $\deg (P) \leq d $ for each $P \in \mathcal{P}$, and
  $\card (\mathcal{P}) =s$. Let $S$ be a
  $\mathcal{P}$-closed semi-algebraic set. Then, for any field of coefficients
  $\F,$
  \begin{eqnarray*}
    b(S,\F) & \leq & \sum_{i=0}^{k}
    \sum_{j=1}^{k-i} \binom{s+1}{j} 6^{j} d (2d-1)^{k-1} .
  \end{eqnarray*}
  If $T$ is a $\mathcal{P}$-semi-algebraic set then, for any field of
  coefficients $\F,$
  \begin{eqnarray*}
    b(T,\F) & \leq & \sum_{i=0}^{k}
    \sum_{j=1}^{k-i} \binom{2s^{2} +1}{j} 6^{j} d (2d-1)^{k-1} .
  \end{eqnarray*}
\end{theorem}

We refer the reader to {\cite{BPR10}} for a survey of other known results in
this direction. Even though the bounds in the case of real varieties often
differ in important respects,
the upper bounds on the Betti numbers in both the real and
complex case share the feature that they depend exponentially in the dimension
of the ambient space, and if the dimension of the ambient space is fixed, of
being polynomial in the degrees of the defining polynomials.

\subsection{Topological complexity of symmetric varieties}
\label{subsec:intro-complexity-symmetric}
Another area of
research with a long history is the action of groups on varieties. Suppose $G$
is a compact group acting on a real or complex variety $V$. If the action is
sufficiently nice then the space of orbits is again a variety in the complex
case and a semi-algebraic set in the real case. Studying the topology of such
orbit spaces is a very natural and well studied problem. We approach it in
this paper from a quantitative point of view, and consider the problem of
proving tight upper bounds on the Betti numbers of the orbit space in terms of
the degrees of the defining polynomials of $V$. In this paper we study
exclusively the orbit spaces of the symmetric group, $\mathfrak{S}_{k}$, or
products of symmetric groups, acting in the standard way on finite dimensional
real or complex vector spaces by permuting coordinates. These orbit spaces
were described (semi-)algebraically in the fundamental papers of Procesi \cite{Procesi78}, and 
Procesi and Schwarz {\cite{Procesi-Schwarz}}.
Subsequently, symmetric group actions in the context of real algebraic
geometry and optimization were studied by several authors (see for example
{\cite{Riener,Timofte03,Timofte05,Timofte05-2,Kuhlmann2012,Blekherman-Riener}}).
We will see that the behavior in terms of topological complexity of the real
and complex orbit spaces differ substantially (unlike in the non-symmetric
situation discussed above).

\begin{notation}
\label{not:multisymmetric-polynomial}
Let $\mathbf{k}= (k_{1} , \ldots ,k_{\omega}) \in \Z_{>0}^{\omega}$, with $k= \sum_{i=1}^{\omega} k_{i}$.
For $P \in \R[\X^{(1)},\ldots,\X^{(\omega)}]$ (resp. $P \in \C[\X^{(1)},\ldots,\X^{(\omega)}]$)
where each  $\X^{(i)}$ is a block of $k_i$ variables.

For $\dd = (d_1,\ldots,d_\omega) \in \Z_{\geq 0}^\omega$, we will denote by 
$\R[\X^{(1)},\ldots,\X^{(\omega)}]_{\leq \dd}$ (resp. $\C[\X^{(1)},\ldots,\X^{(\omega)}]_{\leq \dd}$) denote the set of polynomials
whose degree in $\X^{(i)}$ is bounded by $d_i$ for $1 \leq i \leq \omega$.

We will denote by $\R[\X^{(1)},\ldots,\X^{(\omega)}]^{\mathfrak{S}_{\kk}}$ (resp. $\C[\X^{(1)},\ldots,\X^{(\omega)}]^{\mathfrak{S}_\kk}$)
the set of polynomials which are  fixed under the action of $\mathfrak{S}_{\mathbf{k}}
  =\mathfrak{S}_{k_{1}} \times \cdots \times \mathfrak{S}_{k_{\omega}}$
 acting by independently permuting each block of variables $\X^{(i)}$.
\end{notation}

\begin{notation}
\label{not:multisymmetric-orbit}
  Let $\mathbf{k}= (k_{1} , \ldots ,k_{\omega}) \in
  \Z_{>0}^{\omega}$, with $k= \sum_{i=1}^{\omega} k_{i}$,  and let $X$ be a
  semi-algebraic subset of $\R^{k}$ or a constructible subset of $\C^{k}$,
  such that the product of symmetric groups $\mathfrak{S}_{\mathbf{k}}
  =\mathfrak{S}_{k_{1}} \times \cdots \times \mathfrak{S}_{k_{\omega}}$ act on
  $X$ by independently permuting each block of coordinates. We will denote by
  $X/\mathfrak{S}_{\mathbf{k}}$ the {{\em orbit space\/}} of this action. If
  $\omega =1$, then $k=k_{1}$, and we will denote $\mathfrak{S}_{\mathbf{k}}$
  simply by $\mathfrak{S}_{k}$.
\end{notation}

We recall first the definition of {{\em equivariant cohomology groups\/}} of a
$G$-space for an arbitrary compact Lie group $G$. For $G$ any compact Lie
group, there exists a {{\em universal principal $G$-space\/}}, denoted $E G$,
which is contractible, and on which the group $G$ acts freely on the right. \
The {{\em classifying space\/}} $B G$, is the orbit space of this action, i.e.
$B G= E G/G$.

\begin{definition}
  \label{def:equivariant-cohomology} (Borel construction) Let $X$ be a space
  on which the group $G$ acts on the left (henceforth a $G$-space). Then, $G$ acts diagonally on the
  space $E G \times X$ by $g (z,x) = (z \cdot g^{-1} ,g \cdot x)$. For any
  field of coefficients $\F$, the {{\em  $G$-equivariant cohomology
  groups of $X$ \/}}with coefficients in $\F$, denoted by
  $\HH^{\ast}_{G} (X,\F)$, is defined by
$\HH^{\ast}_{G} (X,\F)  =  \HH^{\ast} (E G \times X/G,\F)$.
\end{definition}

For any $G$-space $X$, there exists a spectral sequence \cite[\S VII.7 (7.2)]{Brown-book} abutting to $\HH^*_G(X,\F)$ whose $E_2$-term is given
by 
\begin{eqnarray*}
\label{eqn:spectral}
E_2^{p,q} = \HH^p(G, \HH^q(X,\F)).
\end{eqnarray*}

The action of $G$ on $X$ induces an action of  $G$ on the 
cohomology ring $\HH^*(X,\F)$, and we denote the subspace of  $\HH^*(X,\F)$ fixed by this action by
$\HH^*(X,\F)^G$.

When $\card(G)$ is invertible in an $\F$-module $M$  (so in particular when $G$ is finite and $\F$ is of characteristic $0$), we have that
$\HH^n(G, M) = 0$, for $n >0$. This implies that when $G$ is finite and $\mathrm{char}(\F) = 0$, 
the spectral sequence \eqref{eqn:spectral} degenerates at its $E_2$-term, and moreover,
\begin{equation}
\label{eqn:iso1}
\HH^n_G(X,\F) \cong \HH^0(G,\HH^n(X,F)) \cong \HH^n(X,\F)^G,
\end{equation}
where the second isomorphism follows from \cite[\S III:1 (1.8)]{Brown-book}.

Moreover, if $X$ is a $G$-space, such that every isotropy group is finite  (for example, when $G$ is finite)  and $\mathrm{char}(\F) = 0$, then 
\begin{equation}
\label{eqn:iso2}
\HH^*(X,\F)^G \cong \HH^*(X/G,\F)
\end{equation}
(see, for example, \cite[page 4, Remark 2]{Brion}).

Thus, combining \eqref{eqn:iso1} and\eqref{eqn:iso2} 
in case $G$ is finite and $\mathrm{char}(\F) = 0$,
 we have the isomorphisms
\begin{equation}
\label{eqn:iso}
\HH^{\ast} (X/G,\F) \xrightarrow{\sim} \HH_{G}^{\ast} (X,\F
) \xrightarrow{\sim} \HH^\ast(X,\F)^G.
\end{equation}

\begin{notation}
  \label{not:equivariant-betti} For any $\mathfrak{S}_{\mathbf{k}}$ symmetric
  semi-algebraic subset $S \subset \R^{k}$ with $\mathbf{k}= (k_{1} , \ldots
  ,k_{\omega}) \in \Z_{>0}^{\omega}$, with $k= \sum_{i=1}^{\omega} k_{i}$,
  and any field $\F$,  we denote
  \begin{eqnarray*}
    b_{\mathfrak{S}_{\mathbf{k}}}^{i} (S,\F) & =  &   \dim_\F \HH^{i}_{\mathfrak{S}_\kk} (S,\F) \\
    b_{\mathfrak{S}_{\mathbf{k}}} (S,\F) & = & \sum_{i \geq 0}
    b_{\mathfrak{S}_{\mathbf{k}}}^{i} (S,\F).
    \end{eqnarray*}
\end{notation}

\begin{remark}
  \label{rem:non-symm-to-symm} Let $\mathbf{k}= (k_{1} , \ldots ,k_{\omega})
  \in \Z_{>0}^{\omega}$, with $k= \sum_{i=1}^{\omega} k_{i}$, and let $V
  \subset \R^{k}$ be a real variety symmetric with respect to the action of
  $\mathfrak{S}_{\mathbf{k}}$ permuting each block of $k_{i}$ coordinates
  independently. Suppose that $V$ is defined by a finite set $\mathcal{P}
  \subset \R [ \X^{(1)} , \ldots
  ,\X^{(\omega)} ]$ of non-negative polynomials
  which are not necessarily symmetric with respect to each block
  $\X^{(i)}$. Then, there exists
  $P^{\symm} \in \R [
  \X^{(1)} , \ldots ,\X^{(
  \omega)} ]$, such that $P^{\symm}$ is symmetric
  in each block $\X^{(i)}$, $\deg (
  P^{\symm}) \leq \max_{P \in \mathcal{P}}   \deg
  (P)$, and $V= \ZZ \left(P^{\symm} , \R^{k}
  \right)$. More precisely, for each $P \in \mathcal{P}$, and each
  $\boldsigma= (\sigma_{1} , \ldots , \sigma_{\omega})
  \in \mathfrak{S}_{\mathbf{k}}$, let
  \begin{eqnarray*}
    P^{\boldsigma} & = & P (\sigma_{1} (
    \X^{(1)}) , \ldots , \sigma_{\omega} (
    \X^{(\omega)}))  ,
  \end{eqnarray*}
  where $\sigma_{i} (\x^{(i)}) = \sigma_{i} (
  \X^{(i)}_{1} , \ldots
  ,\X^{(i)}_{k_{i}}) = (
  \X^{(i)}_{\sigma_{i} (1)} , \ldots
  ,\X^{(i)}_{\sigma_{i} (k_{i})})$ for each $i,1
  \leq i \leq \omega$.
  
  Then, $P^{\boldsigma}$ is also non-negative over
  $\R^{k}$, and $\deg (P^{\boldsigma}) = \deg (P)$.
  Now letting
  \begin{eqnarray*}
    P^{\symm} & = & \sum_{P \in
    \mathcal{P},\boldsigma \in
    \mathfrak{S}_{\mathbf{k}}} P^{\boldsigma}
  \end{eqnarray*}
  we have that $P^{\symm} \in \R [
  \X^{(1)} , \ldots ,\X^{(
  \omega)} ]$, $V= \ZZ \left(P^{\symm} , \R^{k}
  \right)$, $P^{\symm}$ is non-negative over
  $\R^{k}$, $\deg (P^{\symm}) \leq \max_{P \in
  \mathcal{P}}   \deg (P)$, and moreover
  $P^{\symm}$ is symmetric in each block of
  variables $\X^{(i)}$.
  
  Notice that the corresponding statement is not always true over $\C$. For
  example, let $\mathbf{k}= (k)$, and consider the symmetric variety $V_{\C}
  =  \ZZ \left(\mathcal{P}, \C^{k} \right)$ defined by
  \begin{eqnarray*}
    \mathcal{P} & = & \bigcup_{1 \leq i \leq k} \left\{ \prod_{j=1}^{d} (
    X_{i} -j) \right\} ,
  \end{eqnarray*}
  with $d \leq k$.
  
  Note that each polynomial in $\mathcal{P}$ is of degree $d$, but not
  symmetric. Now, $b_{0} \left(V_{\C} /\mathfrak{S}_{k} ,\mathbb{Q} \right)
  =  (\Theta (k))^{d}$ (see Example \ref{ex:zero-dimensional}). On
  the other hand we show (see 
 \eqref{eqn:betti-bound-quotient-complex}) 
  that for any symmetric variety $V_{\C} \subset \C^{k}$ defined by
  symmetric
  polynomials of degree at most $d \leq k$,
  \begin{eqnarray*}
    b_{0} \left(V_{\C} /\mathfrak{S}_{k} ,\mathbb{Q} \right) & \leq & d^{O (
    d)}  .
  \end{eqnarray*}
  This leads to a contradiction for $k \gg d.$ Thus, it is not possible to
  describe $V_{\C}$ by symmetric polynomials in $\C [ X_{1} , \ldots ,X_{k} ]$
  of degree $d$.
\end{remark}

Now let $V=  \ZZ \left(P, \R^{k} \right)$ be a variety that is invariant
under the usual action of $\mathfrak{S}_{\mathbf{k}}$ for some $\mathbf{k}= (
k_{1} , \ldots ,k_{\omega}) \in \Z_{>0}^{\omega}$, with $k=
\sum_{i=1}^{\omega} k_{i}$. A fundamental result due to Procesi and Schwarz
{\cite{Procesi-Schwarz}} states that the orbit space
$V/\mathfrak{S}_{\mathbf{k}}$ has the structure of a semi-algebraic set which
has the following explicit description.

\begin{notation}
  \label{not:elementary-symmetric} For each $k \geq 1  ,i  \geq 0$, we will
  denote by $e_{i}^{(k)} (X_{1} , \ldots ,X_{k})$ the $i$-th elementary
  symmetric polynomial in $X_{1} , \ldots ,X_{k}$, and denote by $\phi_{k} :
  \R^{k} \rightarrow \R^{k}$ (resp., $\phi_{k} :
  \C^{k} \rightarrow \C^{k}$), the map defined by $\x \mapsto (e_{1}^{(k)} (
  \x) , \ldots ,e_{k}^{(k)} (\x))$. Similarly, for $k \geq 1,i  \geq 0$,
  we denote
  \begin{eqnarray*}
    p_{i}^{(k)} (X_{1} , \ldots ,X_{k}) & = & \sum_{j=1}^{k} X_{j}^{i} ,
  \end{eqnarray*}
  and denote by $\psi_{k} : \R^{k} \rightarrow \R^{k}$ (resp., $\psi_{k} : \C^{k} \rightarrow \C^{k}$) , the map defined by
  $\x \mapsto (p_{1}^{(k)} (\x) , \ldots ,p_{k}^{(k)} (
  \x))$. 
  
  \noindent More generally, for $\mathbf{k}= (k_{1} , \ldots
  ,k_{\omega}) \in \Z_{>0}^{\omega}$, with $k= \sum_{i=1}^{\omega} k_{i}$, 
  we will denote by $\phi_{\mathbf{k}} : \R^{k} \rightarrow \R^{k}$
  (respectively $\psi_{\mathbf{k}} : \R^{k} \rightarrow \R^{k}$) the map
  defined by $(\x^{(1)} , \ldots ,\x^{(\omega)}) \mapsto
  (\phi_{k_{1}} (\x^{(1)}) , \ldots , \phi_{k_{\omega}} (
  \x^{(\omega)}))$ (respectively $(\x^{(1)} , \ldots
  ,\x^{(\omega)}) \mapsto (\psi_{k_{1}} (\x^{(1)}) ,
  \ldots , \psi_{k_{\omega}} (\x^{(\omega)}))$). 
  We will also denote by the same symbols, $\phi_\kk,\psi_\kk$,  the corresponding maps $\C^k \rightarrow \C^k$ in the complex case.
  This should not cause any confusion.

  Note that the \emph{Newton identities} (see for example \cite[page 103]{BPRbook2}) give expressions for each sequence of polynomials $(
  e_{i}^{(k)})_{1 \leq i \leq k}$ and $(p_{i}^{(k)})_{1 \leq i \leq k}$
  in terms of the other. Moreover, for all $j \geq 0$, there exists uniquely
  defined polynomials $g_{j}^{(k)} \in \Q [ Z_{1} , \ldots ,Z_{k} ]$ such
  that
  \begin{eqnarray*}
    p_{j}^{(k)} (X_{1} , \ldots ,X_{k}) & = & g_{j}^{(k)} (p_{1}^{(k
   )} , \ldots ,p_{k}^{(k)}) .
  \end{eqnarray*}
  In particular,
  \begin{eqnarray*}
    g_{0}^{(k)} (Z_{1} , \ldots ,Z_{k}) & = & k,\\
    g_{j}^{(k)} (Z_{1} , \ldots ,Z_{k}) & = & Z_{j}  ,1 \leq j \leq k.
  \end{eqnarray*}
  \end{notation}

Note that 
\begin{eqnarray}
\deg(g_j^{(k)}) &\leq & 1 \mbox{ for $0 \leq j \leq k$}, \\
\deg(g_j^{(k)}) &\leq & j \mbox{ for $j > k $}.
\end{eqnarray}
 
\begin{notation}
  \label{not:hankel}We denote by $\Hank^{(k)} (
  Z_{1} , \ldots ,Z_{k}) \in \R [ Z_{1} , \ldots ,Z_{k} ]^{k \times k}$ the
  matrix defined by
  \begin{eqnarray}
  \label{eqn:deg-hankel1}
  \nonumber
    (\Hank^{(k)} (Z_{1} , \ldots ,Z_{k})
   )_{i,j} & = & (ij) g^{(k)}_{i+j-2} (Z_{1} , \ldots ,Z_{k}). 
  \end{eqnarray}
\end{notation}

Note that the degree of $\det(\Hank^{(k)})$ is dominated by the degree of the product of its elements on the main diagonal, and
it follows from \eqref{eqn:deg-hankel1} that,
\begin{eqnarray}
\label{eqn:deg-hankel2}
\nonumber
\deg(\det(\Hank^{(k)})) &\leq & 2(1 + 2 + \cdots + (k-1)) \mbox{ (using \eqref{eqn:deg-hankel1})} \\
&\leq& k(k-1).
\end{eqnarray}

\begin{notation}
  For any real symmetric matrix $A \in \R^{k \times k}$ we denote by $A
  \succeq 0$ the property that $A$ is positive semi-definite. 
\end{notation}

Now suppose that $\mathbf{k}= (k_{1} , \ldots ,k_{\omega}) \in \Z_{>0}, \dd= (d_1,\ldots,d_\omega) \in
\Z_{\geq 0}^{\omega}$, with $k= \sum_{i=1}^{\omega} k_{i}$, and 
$Q \in \mathbf{L}[\X^{(1)} , \ldots ,\X^{(\omega)} ]^{\mathfrak{S}_\kk}_{\leq \dd}$ (cf. Notation \ref{not:multisymmetric-polynomial}),
where $\mathbf{L} = \R$ or $\C$.

\begin{lemma}
\label{lem:multisymmetric-polynomial}
With the notation introduced above, there exists a polynomial $\widetilde{Q} \in \mathbf{L} [\ZB^{(1)} , \ldots ,\ZB^{(\omega)} ]_{\leq \dd}$, such that
\begin{eqnarray*}
&Q (\X^{(1)} , \ldots
   ,\X^{(\omega)})  =&
   \\
   &
   \widetilde{Q} (p_{1}^{(k_{1})} (
   \X^{(1)}) , \ldots ,p_{d}^{(k_{1})} (
   \X^{(1)}) , \ldots ,p_{1}^{(k_{\omega})} (
   \X^{(\omega)}) , \ldots ,p_{d}^{(k_{\omega})}
   (\X^{(\omega)})).& 
   \end{eqnarray*}
\end{lemma}

\begin{proof}
First observe that
\[
\mathbf{L}[\X^{(1)},\ldots,\X^{(\omega)}]^{\mathfrak{S}_\kk}_{\leq \dd} \cong \mathbf{L}[\X^{(1)}]^{\mathfrak{S}_{k_1}}_{\leq d_1} \otimes \cdots \otimes \mathbf{L}[\X^{(\omega)}]^{\mathfrak{S}_{k_\omega}}_{\leq d_\omega},
\]
and for each $i, 1 \leq i \leq \omega$, 
using the fundamental theorem of
symmetric polynomials,
\[
\mathbf{L}[\X^{(i)}]^{\mathfrak{S}_{k_i}} = \mathbf{L}[p_1^{(k_i)}(\X^{(i)}),\ldots, p_{k_i}^{(k_i)}(\X^{(i)})].
\]
The lemma follows immediately.
\end{proof}

Now let    $Q \in \mathbf{L}[\X^{(1)} , \ldots ,\X^{(\omega)} ]^{\mathfrak{S}_\kk}_{\leq \dd}$.
Let $V = \ZZ \left(Q, \mathbf{L}^{k} \right)$, and let $\mathfrak{S}_{\mathbf{k}}$
act on $V$ by permuting each block of coordinates
$\X^{(i)} ,1 \leq i \leq \omega$.
Since for each $k$ the polynomials $p_1^{(k)},\ldots,p_k^{(k)}$ separate the $\mathfrak{S}_k$ orbits in 
$\mathbf{L}^k$,  the image of the map $\psi_{\mathbf{k}}$ 
is homeomorphic to the quotient $V/\mathfrak{S}_{\mathbf{k}}$, 
a fact that we record in the following proposition.

\begin{proposition}
  \label{prop:quotient-sa} The quotient space $V/\mathfrak{S}_{\mathbf{k}}$ is
  homeomorphic to the image $\psi_{\mathbf{k}} (V)$.
\end{proposition}

In the case $\mathbf{L} = \R$,
by the Tarski-Seidenberg principle (see for example \cite[Chapter 2]{BPRbook2}) the image of $\psi_{\mathbf{k}}$ is a semi-algebraic set. Procesi and Schwarz provided the following description 
of the image of $\psi_{\mathbf{k}}$ as a \emph{basic closed} semi-algebraic set.

\begin{theorem}
  \label{thm:description-procesi}{\cite{Procesi-Schwarz}} 
  The image of
  $\psi_{\mathbf{k}}$ is a basic closed semi-algebraic set described by
  \begin{eqnarray}
  \label{eqn:procesi}
   \psi_{\mathbf{k}} (\R^k) &= &\{
    (\mathbf{z}^{(1)} , \ldots ,\mathbf{z}^{(\omega)}) \in \R^{k}   \mid
    \Hank^{(k_{i})} (\mathbf{z}^{(i)})
    \succeq 0,1 \leq i \leq \omega\}. 
  \end{eqnarray}
\end{theorem}

Using the same notation as in Proposition \ref{prop:quotient-sa}, let 
$V = \ZZ(Q,\R^k)$ and $S$ the semi-algebraic set defined by $Q \geq 0$.
We have the following corollary of Theorem  \ref{thm:description-procesi}.

\begin{corollary}
  \label{cor:description-procesi}
  The images $\psi_{\mathbf{k}} (V), \psi_{\mathbf{k}}(S)$ are basic closed semi-algebraic sets 
  described by
  \begin{eqnarray*}
  \nonumber
    & \psi_{\mathbf{k}} (V) = &  \\
    &\ZZ (\widetilde{Q}, \R^{k}) \cap \{
    (\mathbf{z}^{(1)} , \ldots ,\mathbf{z}^{(\omega)}) \in \R^{k}   \mid
    \Hank^{(k_{i})} (\mathbf{z}^{(i)})
    \succeq 0,1 \leq i \leq \omega\},&  \\
    & \psi_{\mathbf{k}} (S) = & \\
    &\RR (\widetilde{Q} \geq 0, \R^{ k}) \cap
    \{ (\mathbf{z}^{(1)} , \ldots ,\mathbf{z}^{(\omega)}) \in
    \R^{k}   \mid  \Hank^{(k_{i})} (
    \mathbf{z}^{(i)}) \succeq 0,1 \leq i \leq \omega \}.&
  \end{eqnarray*}
\end{corollary}

\subsection{Comparison between real and complex quotients}
\label{subsec:intro-comparison}
In order to contrast the topological behavior of the quotient space of equivariant real and complex varieties,
fix two finite sets of polynomials 
\begin{equation}
\label{eqn:PR}
\mathcal{P}_{\R} \subset \R [ X_{1} ,
\ldots ,X_{k} ], \mathcal{P}_{\C} \subset \C [ X_{1} , \ldots ,X_{k}],
\end{equation}
symmetric in $X_{1} , \ldots ,X_{k}$, and let $V_{\R} = \ZZ \left(
\mathcal{P}_{\R} , \R^{k} \right)$ and $V_{\C} = \ZZ \left(\mathcal{P}_{\C} ,
\C^{k} \right)$. Let $\deg (P)   \leq  d \leq k$ for each $P \in
\mathcal{P}_{\R} \cup \mathcal{P}_{\C}$. Let $\mathfrak{S}_{k}$ act on
$V_{\R}$ as well as $V_{\C}$ by permuting the coordinates $X_{1} , \ldots
,X_{k}$.

\subsubsection{Complex quotient}
The quotient space $V_{\C} /\mathfrak{S}_{k}$ is an \emph{algebraic subset} of 
$\C^k$.
To see this we first need a well known result whose proof we include for completeness.
\begin{lemma}
\label{lem:surjective}
The maps $\phi_k,\psi_k:\C^k \rightarrow \C^k$ are surjective.
\end{lemma}

\begin{proof}
First observe that because of Newton identities it suffices to prove the lemma for the map $\phi_k$. Given, $\zb = (z_1,\ldots,z_k) \in \C^k$,
consider the polynomial $F_\zb = T^k - z_1 T^{k-1} + \cdots+ (-1)^d z_d$. Since $\C$ is algebraically closed there exists $k$ roots,
$x_1,\ldots,x_k \in \C$ of $F_\zb$. Then, $\phi_k(x_1,\ldots,x_k) = \zb$.
\end{proof}

Now it follows from the fundamental theorem of symmetric polynomials, that for each $P  \in \mathcal{P}_{\C}$, there exists a polynomial 
$\widetilde{P} \in \C [ Z_{1} , \ldots,Z_{d} ]$ 
with 
$\deg (\widetilde{P}) \leq d$, such that $P=\widetilde{P} (p_{1}^{(k)} ,\ldots ,p_{d}^{(k)})$. 
It then follows from Proposition \ref{prop:quotient-sa} and Lemma \ref{lem:surjective} that

\begin{eqnarray}
\label{eqn:complex-quotient-algebraic}
V_{\C} /\mathfrak{S}_{k} \cong
\ZZ(\widetilde{\mathcal{P}}_\C, \C^{d}) 
\times \C^{k-d}, 
\end{eqnarray}
where $\widetilde{\mathcal{P}}_\C= \bigcup_{P \in \mathcal{P}_{\C}} \{ \widetilde{P} \}$.

It now follows from \eqref{eqn:complex-quotient-algebraic}  and Corollary
\ref{cor:betti-bound-algebraic-complex} that,
with the assumptions above, and for
  any field of coefficients $\F$,
  \begin{eqnarray}
  \label{eqn:betti-bound-quotient-complex}
  \nonumber
    b(V_{\C} /\mathfrak{S}_{k} ,\F) & \leq & 2d (
    4d-1)^{2d-1}\\
    & = & d^{O (d)} .
  \end{eqnarray}

More generally, let $\mathbf{k}= (k_{1} , \ldots ,k_{\omega}) \in
\Z_{>0}^{\omega}, \dd = (d_1,\ldots,d_\omega) \in \Z_{\geq 0}, \dd\leq \kk$, with $k= \sum_{i=1}^{\omega} k_{i}, d= \sum_{i=1}^{\omega}d_i$, 
$\mathcal{P}_{\C}
\subset \C [ \X^{(1)} , \ldots,\X^{(\omega)} ]^{\mathfrak{S}_\kk}_{\leq \dd}$.
Denoting as above $V_{\C} = \ZZ(\mathcal{P}_{\C} , \C^{k})$ we
have:

\begin{theorem}
  For any field of coefficients $\F,$
  \begin{eqnarray*}
    b(V_{\C} /\mathfrak{S}_{\mathbf{k}} ,\F) & \leq
    & 2d (4d -1)^{2d'-1},  
  \end{eqnarray*}
  where $d' =  \sum_{i=1}^{\omega} \min (k_{i},d_i)$. 
  
  In particular, if $d_i \leq k_{i}$ for each $i,1 \leq i \leq \omega$,
  \begin{eqnarray*}
    b(V_{\C} /\mathfrak{S}_{\mathbf{k}} ,\F) & \leq
    & d^{O (\omega  d)} .
  \end{eqnarray*}
\end{theorem}

\begin{proof}
Using Lemma \ref{lem:multisymmetric-polynomial} we have that
for each $P  \in \mathcal{P}_{\C}$, there exists
$\widetilde{P}
\in \C [
\ZB^{(1)} , \ldots ,\ZB^{(d)} ]_{\leq \dd}$, where for each $i,1
\leq i \leq \omega$, $\ZB^{(i)}$ is a block of $\min (k_i,d_i)$ variables, such that
\begin{eqnarray*}
&P (\X^{(1)} , \ldots ,\X^{(\omega)}) =& \\
& \widetilde{P}
(
   p_{1}^{(k_{1})} (\X^{(1)}) , \ldots ,p_{\ell_{1}}^{(k_{1}
  )} (\X^{(1)}) , \ldots ,p_{1}^{(k_{\omega})} (
   \X^{(\omega)}) , \ldots ,p_{\ell_{\omega}}^{(k_{\omega})} (
   \X^{(\omega)})). &
\end{eqnarray*}

The quotient space, $V_{\C} /\mathfrak{S}_{\kk}$, is then 
isomorphic
to 
$\ZZ(\widetilde{\mathcal{P}}_\C, \C^{d'}) 
\times \C^{k-d'}
$, where $\widetilde{\mathcal{P}}_\C= \bigcup_{P \in \mathcal{P}_{\C}} \{ \widetilde{P} \}$. 
Now apply 
Corollary
\ref{cor:betti-bound-algebraic-complex}.
\end{proof}

This shows in particular, that in case $d_i \leq k_{i}$ for each $i$, the Betti
numbers of the quotient space $V_{\C} /\mathfrak{S}_{\mathbf{k}}$ can be
bounded in terms of $d$ and $\omega$, independent of $k$.

\subsubsection{Real quotient}
In contrast, the space of orbits of the action of
$\mathfrak{S}_{\mathbf{k}}$ on $V_{\R}$ has the structure of a semi-algebraic
(rather than an algebraic) set (see Proposition \ref{prop:quotient-sa} above).
It is also not possible to bound $b \left(V_{\R} /\mathfrak{S}_{\mathbf{k}}
,\F \right)$ by a function of $\omega$ and $d$ independent of $k$
(similar to the complex case) as shown by the following example.

\begin{example}
  \label{ex:zero-dimensional}Let  $\mathbf{k}= (k)$, and
  \begin{eqnarray*}
    P & = & \sum_{i=1}^{k} \left(\prod_{j=1}^{d} (X_{i} -j) \right)^{2}.
  \end{eqnarray*}
  Then $P$ is symmetric of degree $2d.$ Let $V_{\R} = \ZZ \left(\{ P \} ,
  \R^{k} \right)$. Then $V_{\R}$ consists  of all points $x\in\{1,\ldots, d\}^k$,
   $V_{\R} /\mathfrak{S}_{k}$ is zero-dimensional, and each orbit is represented by 
   a point $\mathbf{y}=(y_1,\ldots,y_k)$, with $1 \leq y_1 \leq y_2 \cdots \leq y_k \leq d$.
   Since each $y_i \in \{1,\ldots,d\}$, the set of orbits  is in one-to-one correspondence with the finite set
   $O_{d,k} = \{(\ell_1,\ldots,\ell_d) \in \Z_{\geq 0} \mid \sum_{i=1}^{d} \ell_i = k\}$.
   It is easy to see that $\card(O_{d,k}) = \binom{d+k-1}{d-1}$.
  Therefore,
  \begin{eqnarray*}
    b_{0} (V_{\R} /\mathfrak{S}_{k} ,\mathbb{Q}) & = &  \binom{d+k-1}{d-1} \\
     &=& (\Theta (k))^{d-1} .
  \end{eqnarray*}
\end{example}

Example \ref{ex:zero-dimensional} shows that there is a fundamental difference
in the topological complexity of the orbit space in the complex and real case.
In the complex case the topological complexity of the orbit space, $V_{\C}
/\mathfrak{S}_{k}$, measured by the sum of the Betti numbers, is bounded by
a function of $d$ independent of $k$ (for $k \geq d$). However, in the real
case, the topology of the space of orbits, $V_{\R} /\mathfrak{S}_{k}$, can
grow with $k$ for fixed $d$. However, it is still possible to bound the
Betti numbers of the quotient $V_{\R} /\mathfrak{S}_{k}$ using the description
of $V_{\R} /\mathfrak{S}_{k}$ given in Theorem \ref{thm:description-procesi},
and the bound on the Betti numbers of basic closed semi-algebraic sets in
Theorem \ref{thm:betti-bound-sa}. 

Let $Q = \sum_{P \in \mathcal{P}_\R} P^2$ (where $\mathcal{P}_\R$ is as in \eqref{eqn:PR}). 
Then there exists using the fundamental theorem of 
symmetric polynomials, 
$\widetilde{Q} \in \R [ Z_{1} , \ldots,Z_{d} ]$ 
with 
$\deg (\widetilde{Q}) \leq 2d$, such that $Q=\widetilde{Q} (p_{1}^{(k)} ,\ldots ,p_{d}^{(k)})$. 

Also notice that a symmetric matrix $A
\in \R^{k \times k}$ is positive semi-definite if and only if all its
symmetric minors are non-negative. 

We can thus describe the set  $\psi_{k}(V_{\R})$ using Eqn. \eqref{eqn:procesi}  
involving
$2^{k}$
polynomial inequalities whose maximum degree equals
\begin{eqnarray*}
  \deg (\det (\Hank^{(k)} (
  \ZB))) & \leq & 
  k (k-1) 
  \mbox{ (using \eqref{eqn:deg-hankel2})},
\end{eqnarray*}
as well as the inequality $-\tilde{Q} \geq 0$.
Applying Theorem \ref{thm:betti-bound-sa} directly (and noting that 
$\deg (\widetilde{Q}) \leq 
2d$), we get for any field of coefficients
$\F$,
\begin{eqnarray*}
   b(\psi_{k} (V_{\R}) ,\F) & \leq & 
  (2^k+1)d'(2(2^k+1)d'+1)^{k-1},
  \end{eqnarray*}
  where $d' = \max(k(k-1),2d)$.
  This yields the bound
  \begin{eqnarray}
   \label{eqn:betti-bound-procesi}
  b(\psi_{k} (V_{\R}) ,\F) & \leq & 
  (O (2^{k} k^2 d))^{k}. 
\end{eqnarray}

An alternative method for bounding the Betti numbers of $V_{\R}
/\mathfrak{S}_{k}$ is to use the ``descent spectral sequence'' argument as in
{\cite{GVZ04}} (see also \cite{Houston2}). Using the fact that the map $\psi_{k}$ is proper {one can construct a spectral sequence which converges to $\HH^\ast(\psi_k(V_{\R}),\F)$. Bounding the dimension of the first term of this sequence then yields  the
inequality that for each $n \geq 0$,
\begin{eqnarray}
\label{eqn:descent}
  b_{n} (\psi_{k}(V_{\R}) ,\F) & \leq &
  \sum_{p+q=n} b_{q} (W^{(p)} ,\F),
\end{eqnarray}
where 
$W^{(p)} = \underbrace{V_{\R} \times_{\psi_{k}} \cdots \times_{\psi_{k}} V_{\R}}_{p+1}$ is the $(p+1)$-fold fibred product
(fibred over the map $\psi_k$) described by
\[
W^{(p)} = \{ (x^0,\ldots,x^p) \in V_\R^{p+1} \mid \psi_k(x^0) = \cdots = \psi_k(x^p)\}.
\]

Clearly, 
$W^{(p)} \subset \R^{(p+1)k} $ 
is defined by 
$(p+1)$ polynomial equations each of degree at most $d$, and $kp$ polynomial equations each of degree at most $k$.
Using inequality \eqref{eqn:descent} and Theorem \ref{thm:betti-bound-algebraic}, we obtain

\begin{eqnarray}
\label{eqn:betti-bound-descent}
\nonumber
  b(\psi_{k} (V_{\R} ) ,\F) & \leq & 
  \sum_{i=0}^{k-1} \sum_{j=0}^{i} b_j(W^{(i-j)},\F) \\
  \nonumber
  & \leq & \sum_{p=0}^{k-1}b(W^{(p)},\F) \\
  \nonumber
  &\leq & \sum_{p=0}^{k-1} \max(2d,k)(2\max(2d,k)+1)^{(p+1)k -1}\\
  &=&
  (k+d)^{O (k^{2})},  
\end{eqnarray}
which is again exponential in $k$ for any fixed $d$. It is also possible to
obtain a bound of a similar shape as in \eqref{eqn:betti-bound-descent}
using a different method. 
First use effective quantifier elimination to
obtain a semi-algebraic description of $\psi_{k}(V_{\R})$, and
then use
Theorem \ref{thm:betti-bound-sa-general}.

\section{Main results and outline of proofs}
\label{sec:main-results}
\subsection{Bounds on equivariant Betti numbers}
Before stating the main theorems of this paper we introduce some more
notation.

\begin{notation}
  \label{not:partition}(Partitions) We denote by $\Pi_{k}$ the set of
  partitions of $k$, where each partition $\pi = (\pi_{1} , \pi_{2} , \ldots
  , \pi_{\ell}) \in \Pi_{k}$, where $\pi_{1} \geq \pi_{2} \geq \cdots \geq
  \pi_{\ell} \geq 1$, and $\pi_{1} + \pi_{2} + \cdots + \pi_{\ell} =k$. We
  call $\ell$ the length of the partition $\pi$, and denote
  $\length (\pi) = \ell$. For 
  $\ell > 0$
  we
  will denote
  \begin{eqnarray*}
    \Pi_{k, \ell} & = & \{ \pi \in \Pi_{k} \mid
    \length (\pi) \leq \ell \} ,\\
    p (k, \ell) & = & \card (\{ \pi \in \Pi_{k}
    \mid \length (\pi) = \ell \}) .
  \end{eqnarray*}
  More generally, for any tuple $\mathbf{k}= (k_{1} , \ldots ,k_{\omega})
  \in \Z_{>0}^{\omega}$, we will denote by
  $\boldPi_{\mathbf{k}} = \Pi_{k_{1}} \times \cdots
  \times \Pi_{k_{\omega}}$, and for each $\boldpi= (
  \pi^{(1)} , \ldots , \pi^{(\omega)}) \in
  \boldPi_{\mathbf{k}}$, we denote by
  $\length (\boldpi) =
  \sum_{i=1}^{\omega} \length (\pi^{(i)})$. We
  also denote for each $\boldsymbol{\ell}= (\ell_{1} , \ldots , \ell_{\omega}) \in
  \Z_{>0}^{\omega}$,
  \begin{eqnarray*}
    | \boldsymbol{\ell} | & = & \ell_{1} + \cdots + \ell_{\omega ,}\\
    \boldPi_{\mathbf{k},\boldsymbol{\ell}} & = & \{
    \boldpi= (\pi^{(1)} , \ldots , \pi^{(\omega)})
    \mid \pi^{(i)} \in \Pi_{k_{i} , \ell_{i}} ,   1 \leq i \leq \omega \}
    ,\\
    p (\mathbf{k},\boldsymbol{\ell}) & = & \card  (\{
    \boldpi= (\pi^{(1)} , \ldots , \pi^{(\omega)})
    \mid \length (\pi^{(i)}) = \ell_{i} ,  1
    \leq i \leq \omega \}) .
  \end{eqnarray*}
\end{notation}

We prove the following theorem.

\begin{theorem}
  \label{thm:main} Let $\mathbf{k}= (k_{1} , \ldots ,k_{\omega}) \in
  \Z_{>0}^{\omega}$,with  $k= \sum_{i=1}^{\omega} k_{i}$. Let $P \in \R [
  \X^{(1)} , \ldots
  ,\X^{(\omega)} ]$, where each
  $\X^{(i)}$ is a block of $k_{i}$ variables, be a
  non-negative polynomial, such that $V= \ZZ \left(P, \R^{k} \right)$ is
  invariant under the action of $\mathfrak{S}_{\mathbf{k}}$ permuting each
  block $\X^{(i)}$ of $k_{i}$ coordinates. Let
 $\deg (P)   \leq  d$.
  Then, for any
  field of coefficients $\F$,
  \begin{eqnarray}
  \label{eqn:thm:main0}
    b (V/\mathfrak{S}_{\mathbf{k}} ,\F) & \leq &
    \sum_{\substack{\boldsymbol{\ell}= (\ell_{1} , \ldots , \ell_{\omega}),\\1 \leq \ell_{i}
    \leq \min (k_{i},2d)}} p (\mathbf{k},\boldsymbol{\ell})  d (2d-1)^{|
    \boldsymbol{\ell} | +1} .
  \end{eqnarray}
  Moreover, for all $i  \geq   \sum_{j=1}^{\omega} \min (k_{j},2d)$
  \begin{eqnarray}
    b_{i} (V/\mathfrak{S}_{\mathbf{k}} ,\F) & = & 0. 
    \label{eqn:thm:main1}
  \end{eqnarray}
  If for each $i,1 \leq i \leq \omega$, $2d  \leq k_{i}$, then
  \begin{eqnarray*}
    b (V/\mathfrak{S}_{\mathbf{k}} ,\F) & \leq & (k_{1}
    \cdots k_{\omega})^{2d} (O (d))^{2 \omega d+1} .
  \end{eqnarray*}
  In particular, in the case 
  $\F = \Q$,
  \begin{eqnarray}
  \label{eqn:thm:main2}
    b_{\mathfrak{S}_{\mathbf{k}}} (V,\mathbb{Q}) & \leq &
    \sum_{\substack{\boldsymbol{\ell}= (\ell_{1} , \ldots , \ell_{\omega}), \\1 \leq \ell_{i}
    \leq \min (k_{i},2d)}} p (\mathbf{k},\boldsymbol{\ell})  d (2d-1)^{|
    \boldsymbol{\ell} | +1} .
  \end{eqnarray}
\end{theorem}

\begin{remark}
  \label{rem:partition-function}For $d=o (k^{1/3})$, and $k \gg 1$, we have
  that $p (k,d) \sim \dfrac{\binom{k-1}{d-1}}{d!} = (\Theta (k))^{d-1}$
  {\cite{Erdos-Lehner}}. Thus, in the special case, when $\omega =1$, $d=O (1
 )$, we have the following asymptotic (for $k  \gg 1$) form of the bound in
  Theorem \ref{thm:main},
  \begin{eqnarray*}
    b (V/\mathfrak{S}_{k} ,\F) & \leq & O (k^{2d-1}) .
  \end{eqnarray*}
\end{remark}

\begin{remark}
\label{rem:multiplicities}
As observed previously (see \eqref{eqn:iso1}),
the action of $\mathfrak{S}_{\mathbf{k}}$ on $V$ induces an action of  $\mathfrak{S}_{\mathbf{k}}$ on the 
cohomology ring $\HH^*(V,\Q)$, and it follows from \eqref{eqn:iso} that there is an isomorphism 
\[
\HH^*(V/\mathfrak{S}_{\mathbf{k}},\Q) \xrightarrow{\sim} \HH^*(V,\Q)^{\mathfrak{S}_{\mathbf{k}}}.
\]

Thus, the bound in \eqref{eqn:thm:main2} gives a \emph{polynomial bound} (for every fixed $d$ and $\omega$) on the \emph{multiplicity} of the trivial representation of
$\mathfrak{S}_{\mathbf{k}}$ in the $\mathfrak{S}_{\mathbf{k}}$-module $\HH^*(V,\Q)$. It is interesting to ask for similar bounds on the multiplicities of other non-trivial irreducible representations of
$\mathfrak{S}_{\mathbf{k}}$ in  $\HH^*(V,\Q)$, and to characterize those that could occur with positive multiplicities.
We will address these questions in a subsequent paper.
\end{remark}

A special case of 
inequality
\eqref{eqn:thm:main0} in Theorem \ref{thm:main}
is of independent interest later. We note this as a corollary.  
\begin{corollary}\label{cor:main}
Suppose that
$\mathbf{k}= (\underbrace{1, \ldots 1}_{m} ,k)$, and $2 \leq d \leq k/2$.
Then, with the same notation as in Theorem \ref{thm:main} above the following bounds hold:
  \begin{eqnarray*}
    b (V/\mathfrak{S}_{\mathbf{k}} ,\F) & \leq & \sum_{1 \leq
    \ell \leq 2d} p (k, \ell)  d (2d-1)^{m+ \ell +1} .\\
    & = & k^{2d}  O (d)^{m+2d+1} .
  \end{eqnarray*}
  \begin{eqnarray*}
    b_{\mathfrak{S}_{\mathbf{k}}} (V,\mathbb{Q}) & \leq & \sum_{1 \leq
    \ell \leq 2d} p (k, \ell)  d (2d-1)^{m+ \ell +1} .\\
    & = & k^{2d}  O (d)^{m+2d+1} .
  \end{eqnarray*}
\end{corollary}
\begin{proof}
Since $\mathbf{k}= (\underbrace{1, \ldots 1}_{m} ,k)$ directly implies $\ell_1=\ldots=\ell_m=1$ the bound is immediate from \eqref{eqn:thm:main0}.   
\end{proof}

\begin{remark}
  \label{rem:poly-vs-exponential}Notice that for fixed $m$ and $d$ both
  bounds in Corollary \ref{cor:main} are polynomial in $k$ compared to the
  bounds in the inequalities \eqref{eqn:betti-bound-procesi} and
  \eqref{eqn:betti-bound-descent} above,  where the dependence on $k$ is singly
  exponential.
\end{remark}

More generally, for symmetric \emph{semi-algebraic} sets we have the
following two theorems (for $\mathcal{P}$-closed semi-algebraic and
$\mathcal{P}$-semi-algebraic sets, respectively).
\begin{notation}
  \label{not:fkd} Let $\mathbf{k}= (k_{1} , \ldots ,k_{\omega}) \in
  \Z_{>0}^{\omega}$, with $k= \sum_{i=1}^{\omega} k_{i}$,  and $d \geq 1$.
  We denote
  \begin{eqnarray*}
    F (\mathbf{k},d) & = & \sum_{\substack{\boldsymbol{\ell}= (\ell_{1} , \ldots ,
    \ell_{\omega}), \\1 \leq \ell_{i} \leq \min (k_{i},2d)}} p (
    \mathbf{k},\boldsymbol{\ell})  d (2d-1)^{| \boldsymbol{\ell} | +1} .
  \end{eqnarray*}
\end{notation}

\begin{theorem}
  \label{thm:main-sa-closed}Let $\mathbf{k}= (k_{1} , \ldots ,k_{\omega})
  \in \Z_{>0}^{\omega}$, with $k= \sum_{i=1}^{\omega} k_{i}$, and let
  $\mathcal{P} \subset \R [ \X^{(1)} , \ldots
  ,\X^{(\omega)} ]$ be a finite set of polynomials,
  where each $\X^{(i)}$ is a block of $k^{(i)}$
  variables, and such that each $P \in \mathcal{P}$ is symmetric in each block
  of variables $\X^{(i)}$. Let $S \subset \R^{k}$
  be a $\mathcal{P}$-closed-semi-algebraic set. Suppose that $\deg (P)  
  \leq  d$ for each $P  \in  \mathcal{P}$, $\card (
  \mathcal{P}) =s$, and let $D=D (\mathbf{k},d) = \sum_{i=1}^{\omega} \min
  (k_{i} ,5d)$. Then, for any field of coefficients $\F$,
  \begin{eqnarray*}
    b (S/\mathfrak{S}_{\mathbf{k}} ,\F) & \leq & \sum_{i=0}^{D-1}
    \sum_{j=1}^{D-i} \binom{2 s+1}{j} 6^{j} F (\mathbf{k},2d)
  \end{eqnarray*}
  (where $F$ is as in Notation \ref{not:fkd}), and moreover
  \begin{eqnarray*}
    b_{i} (S/\mathfrak{S}_{\mathbf{k}} ,\F) & = & 0  ,
  \end{eqnarray*}
  for $i  \geq  D$.
\end{theorem}

\begin{remark}
  In the particular case, when $\omega =1$, $d=O (1)$, and $k \gg 1$,  $D = \min(k,5d) = 5d$, 
  $F(\kk,2d) = (O(k))^{4d-1}$ (using the definition given in Notation \ref{not:fkd} and Remark \ref{rem:partition-function}),  and
  the bound in Theorem \ref{thm:main-sa-closed} takes the following asymptotic
  form:
  \begin{eqnarray*}
    b (S/\mathfrak{S}_{k} ,\F) & \leq & 
    s^{5d-1} (O(k))^{4d-1} .
  \end{eqnarray*}
\end{remark}

For general $\mathcal{P}$-semi-algebraic sets we have:

\begin{theorem}
  \label{thm:main-sa}Let $\mathbf{k}= (k_{1} , \ldots ,k_{\omega}) \in
  \Z_{>0}^{\omega}$, with $k= \sum_{i=1}^{\omega} k_{i}$, and let
  $\mathcal{P} \subset \R [ \X^{(1)} , \ldots
  ,\X^{(\omega)} ]$ be a finite set of polynomials,
  where each $\X^{(i)}$ is a block of $k^{(i)}$
  variables, and such that each $P \in \mathcal{P}$ is symmetric in each block
  of variables $\X^{(i)}$. Let $S \subset \R^{k}$
  be a $\mathcal{P}$-semi-algebraic set. Suppose that $\deg (P)   \leq  d$
  for each $P  \in  \mathcal{P}$, $\card (
  \mathcal{P}) =s$ and let $D=D (\mathbf{k},d) = \sum_{i=1}^{\omega} \min (
  k_{i} ,5d)$. Then, for any field of coefficients $\F$, 
  \begin{eqnarray*}
    b(S/\mathfrak{S}_{k} ,\F) & \leq & \sum_{i=0}^{D-1}
    \sum_{j=1}^{D-i} \binom{8 (k+1) (s+1)}{j} 6^{j} F (\mathbf{k},2d),
  \end{eqnarray*}
  and
  \begin{eqnarray*}
    b_{i} (S/\mathfrak{S}_{k} ,\F) & = & 0  ,
  \end{eqnarray*}
  for $i  \geq  D$.
\end{theorem}

\begin{remark}
  In the particular case, when $\omega =1$, $d=O (1)$,  and $k \gg 1$, the bound in Theorem
  \ref{thm:main-sa} takes the following asymptotic  form.
  \begin{eqnarray*}
    b (S/\mathfrak{S}_{k} ,\F) & \leq & s^{5d} k^{O (d)} .
  \end{eqnarray*}
\end{remark}

\begin{remark}(Tightness)
Example \ref{ex:zero-dimensional} shows that the sum of the equivariant Betti numbers of 
a symmetric real algebraic set $V\subset \R^k$, defined by symmetric polynomials of degree
at most $d$ could be as large as $k^{\Theta(d)}$. It is not too difficult to also to show that 
in the case of a symmetric $\mathcal{P}$-semi-algebraic set, the dependence on 
$s = \mathrm{card}(\mathcal{P})$ can be of the order of $s^{\Theta(d)}$
where $d = \max_{P\in \mathcal{P}} \mathrm{deg}(P)$. 

To see this consider the semi-algebraic set  $\psi_{k,d}(\R^k)$, where $\psi_{k,d} = \pi_d \circ \psi_k$, and 
$\psi_k$ is defined in Notation \ref{not:elementary-symmetric} and $\pi_d$ is the projection to the first $d$ coordinates.
Since  $\psi_k(\R^k)$ has dimension $k$ (using Proposition \ref{prop:quotient-sa} with $V = \R^k$), $\psi_{k,d}(\R^k)$
is of dimension $d$, and thus has non-empty interior. Let $z = (z_1,\ldots,z_d) \in \R^d$ belong to the interior of 
$\psi_{k,d}(\R^k)$. Then, it is easy to see that there exists a set  
$\widetilde{\mathcal{P}} \subset \R[Z_1,\ldots,Z_d]$ of $s$ linear polynomials,
such that in a closed ball 
\[
\widetilde{B} = \overline{B_d(z,\eps)} \subset \psi_{k,d}(\R^k)    \mbox{ (cf. Notation \ref{not:ball})},
\]
with $\eps >0$ and small enough,
$$\displaylines{
\widetilde{S} := \widetilde{B}  \setminus \bigcup_{P \in \widetilde{P}} \ZZ(P,\R^d)
}
$$
has $(\Omega(s))^d$ connected components. 
It is then clear that defining 
$$\displaylines{
\mathcal{P} = \bigcup_{\widetilde{P} \in \widetilde{\mathcal{P}}} \{\widetilde{P}(p_1^{(k)},\ldots,p_d^{(k)}) \},
}
$$
the symmetric semi-algebraic set
$$\displaylines{
S = B \setminus (\bigcup_{P \in \mathcal{P}}  \ZZ(P,\R^k)),
}
$$
where $B$ is defined by
\[
\sum_{i=1}^{d} (p_i^{(k)} - z_i)^2 - \eps  \leq 0,
\] 
has the property that,
\[
\psi_{k}(S) = \pi_d^{-1}(\widetilde{S}) \cap \psi_{k}(\R^k),
\]
and hence using Proposition \ref{prop:quotient-sa} that,
\[
b_0(S/\mathfrak{S}_k,\F) \geq b_0(\widetilde{S},\F) =  \Omega(s)^d
\]
(actually, the first inequality is an equality, but we do not need this fact for the lower bound).
 
Notice that  $S$ is a $\mathcal{P}'$-semi-algebraic set where 
\[
\mathcal{P}' = \mathcal{P} \cup  \{\sum_{i=1}^{d} (p_i^{(k)} - z_i)^2 - \eps\},
\] 
and hence $\card(\mathcal{P}') = s+1$, and the maximum
degree of the polynomials in $\mathcal{P}'$ is bounded by $2d$.

Hence, the bounds in 
Theorems \ref{thm:main}, \ref{thm:main-sa-closed} and \ref{thm:main-sa} are asymptotically
tight for fixed $d$ and $s,k$ large.
\end{remark}

\subsection{An application in a non-equivariant
setting}
\label{subsec:non-equivariant} 
As an application of Theorems
\ref{thm:main} and Theorem \ref{thm:main-sa-closed}, we obtain an
improvement in certain situations of a result of Gabrielov, Vorobjov and Zell
{\cite{GVZ04}} bounding the Betti numbers of a semi-algebraic set described as
the projection of another semi-algebraic set in terms of the description
complexity of the pre-image. 
This improvement is relevant for bounding the Betti numbers of the images of general (not necessarily symmetric) semi-algebraic sets under certain proper maps, and thus is an application of the main results of this paper in a non-equivariant setting.

Let $\mathcal{P} \subset \R [ Y_{1} , \ldots ,Y_{m} ,X_{1} , \ldots ,X_{k}
]$ be a family of polynomials and with $\deg (P) \leq d,P \in \mathcal{P}$,
$\card (\mathcal{P}) =s$. Let $\pi : \R^{m+k}
\rightarrow \R^{m}$ be the projection map to the first $m$ co-ordinates,
and let $S$ be a bounded $\mathcal{P}$-closed semi-algebraic set. We
consider the problem of bounding the Betti numbers of the image $\pi (S)$.
There are two different approaches. One can first obtain a semi-algebraic
description of the image $\pi (S)$ with bounds on the degrees and the number
of polynomials appearing in this description and then apply known bounds on
the Betti numbers of semi-algebraic sets in terms of these parameters. Another
approach is to use the ``descent spectral sequence'' of the map $\pi |_{S}  
$which abuts to the cohomology of $\pi (S)$, and bound the Betti numbers of
$\pi (S)$ by bounding the dimensions of the $E^{1}$-terms of this spectral
sequence. 
For this approach it is 
important
that 
the map $\pi$ is proper (which is ensured by requiring that the set $S$ is closed and bounded)
since in the general case the spectral sequence might not converge to $\HH^*(S,\F)$. 
The second approach produces a slightly better bound. The following
theorem whose proof uses the second approach appears in {\cite{GVZ04}}.

\begin{theorem}
  \label{thm:descent-quantitative} 
  Let  
$S\subset\R^{m+k}$
be a closed and bounded semi-algebraic set. Then with the same notation as above,
  \begin{eqnarray*}
    b (\pi (S) ,\F) & = & (O (s d))^{(k+1) m} .
  \end{eqnarray*}
\end{theorem}

In the special case when $k=1$, Theorem \ref{thm:descent-quantitative} implies
that
\begin{eqnarray}
  b (\pi (S) ,\F) & = & (O (s d))^{2m} . 
  \label{eqn:projection-bound}
\end{eqnarray}
\begin{remark}
  Notice, that the coefficient $2$ in the exponent in the bound above is
  present even if one uses the first approach of using effective quantifier
  elimination. In this case, the exponent $2m$ occurs due to the fact that the
  sub-resultants (with respect to the variable $X_{1}$) of two polynomials
  $P_{1} ,P_{2} \in \mathcal{P}$ can have degree as large as $d (d-1) =O (
  d^{2})$ in the variables $Y_{1} , \ldots ,Y_{m}$, and moreover the $O (
  s^{2})$ such sub-resultants are used in the description of $\pi (S)$ (see
  for example the complexity analysis of Algorithm 14.1 in {\cite{BPRbook2}}).
  As a result the exponent in the bound on the Betti numbers of $\pi (S)$
  obtained through this method is again $2m$. Note that the squaring of the
  degree and the number of polynomials involved are responsible for the doubly
  exponential complexity of quantifier elimination in the first order theory
  of real closed fields -- and seems unavoidable if one wants to describe the
  image of a projection.
\end{remark}

As a consequence of the main result of this paper, we obtain the following
bound on the Betti numbers of the image under projection to one less
dimension of real algebraic varieties
(not necessarily symmetric).

\begin{theorem}
  \label{thm:descent2-quantitative}Let $P \in \R [ Y_{1} , \ldots ,Y_{m} ,X
  ]$ be a non-negative polynomial and with $\deg (P) \leq d$. Let $V= \ZZ
  \left(P, \R^{m+1} \right)$ be bounded, and $\pi : \R^{m} \times \R
  \rightarrow\R^{m}$ be the projection map to the first $m$ coordinates.
  For each $p  ,0 \leq p<m$, let $\mathbf{k}_{m,p} = (\underbrace{1, \ldots
  ,1}_{m} ,p)$. Then,
  \begin{eqnarray*}
    b (\pi (V) ,\F) & \leq & \sum_{0 \leq p<m} F (\mathbf{k}_{m,p}
    ,d) = 
    {m}^{2d} (O (d))^{m+2d+1}.
  \end{eqnarray*}
\end{theorem}

Theorem \ref{thm:descent2-quantitative} 
yields better asymptotic bounds compared to
the bound in \eqref{eqn:projection-bound} above, when $d$ is held fixed, and
$m \rightarrow \infty$.

\subsection{Outline of the proofs of the main theorems}
\label{subsec:outline}
Most bounds on the
Betti numbers of real algebraic varieties are usually proved by first making a
deformation to a set defined by one inequality with smooth boundary and
non-degenerate critical points with respect to some affine function.
Furthermore, the new set is homotopy equivalent to the given variety and it
thus suffices to bound the Betti numbers of its boundary (up to a
multiplicative factor of $2$). Finally, the last step is accomplished by
bounding the number of critical points using the Bezout bound. The approach
used in this paper for bounding the equivariant Betti numbers is somewhat
similar. However, since the perturbation, as well as the Morse function both
need to be equivariant, the choices are more restrictive (see Proposition
\ref{prop:non-degenerate}). Additionally, the topological changes at the
Morse critical points need to be analyzed more carefully (see Lemmas
\ref{lem:equivariant_morseA} and \ref{lem:equivariant_morseB}). The main
technical tool that makes the good dependence on the degree $d$ of the
polynomial possible is the so called ``half-degree principle''
{\cite{Riener,Timofte03}} (see Lemma \ref{lem:half-degree} as well as
Proposition \ref{prop:half-degree}), and this is what we use rather than
the Bezout bound to bound the number of (orbits of) critical points. The
semi-algebraic case as usual provides certain additional obstacles. We adapt
the techniques developed in {\cite[Chapter 7]{BPRbook2}} to the equivariant
situation to reduce to the (equivariant) algebraic case. The main tool used
here are certain inequalities coming from the Mayer-Vietoris exact sequence.
Finally, for the proof of Theorem \ref{thm:descent2-quantitative} 
we extend to the equivariant setting the
descent spectral sequence defined in {\cite{GVZ04}}. The role of the fibered
join used in {\cite{GVZ04}} is now replaced by the fibered symmetric join (see
Theorem \ref{thm:symmetric-spectral-sequence}). We prove the necessary
topological properties of the symmetric join (see Lemma
\ref{lem:contractible-infinite-join}, Proposition \ref{prop:he} and Lemma
\ref{lem:quotient}). The proof of Theorem
\ref{thm:descent2-quantitative} 
then consists of applying Theorem  \ref{thm:main} 
to  bound the $E^{1}$-term of this new spectral sequence defined in Theorem
\ref{thm:symmetric-spectral-sequence}.

\section{Background and preliminaries}\label{sec:prelim}

In this section we recall some basic facts about real closed fields and real
closed extensions.

\subsection{Real closed extensions and Puiseux series}We will need some
properties of Puiseux series with coefficients in a real closed field. We
refer the reader to {\cite{BPRbook2}} for further details.

\begin{notation}
  For $\R$ a real closed field we denote by $\R \left\langle \eps
  \right\rangle$ the real closed field of algebraic Puiseux series in $\eps$
  with coefficients in $\R$. We use the notation $\R \left\langle \eps_{1} ,
  \ldots , \eps_{m} \right\rangle$ to denote the real closed field $\R
  \left\langle \eps_{1} \right\rangle \left\langle \eps_{2} \right\rangle
  \cdots \left\langle \eps_{m} \right\rangle$. Note that in the unique
  ordering of the field $\R \left\langle \eps_{1} , \ldots , \eps_{m}
  \right\rangle$, $0< \eps_{m} \ll \eps_{m-1} \ll \cdots \ll \eps_{1} \ll 1$.
\end{notation}

\begin{notation}
\label{not:lim}
  For elements $x \in \R \left\langle \eps \right\rangle$ which are bounded
  over $\R$ we denote by $\lim_{\eps}  x$ to be the image in $\R$ under the
  usual map that sets $\eps$ to $0$ in the Puiseux series $x$.
\end{notation}

\begin{notation}
\label{not:extension}
  If $\R'$ is a real closed extension of a real closed field $\R$, and $S
  \subset \R^{k}$ is a semi-algebraic set defined by a first-order formula
  with coefficients in $\R$, then we will denote by $\Ext(S, \R') \subset \R'^{k}$ the semi-algebraic subset of $\R'^{k}$ defined by
  the same formula. It is well-known that $\Ext(S, \R')$ does
  not depend on the choice of the formula defining $S$ {\cite{BPRbook2}}.
\end{notation}

\begin{notation}
\label{not:ball}
  For $x \in \R^{k}$ and $r \in \R$, $r>0$, we will denote by $B_{k} (x,r)$
  the open Euclidean ball centered at $x$ of radius $r$. If $\R'$ is a real
  closed extension of the real closed field $\R$ and when the context is
  clear, we will continue to denote by $B_{k} (x,r)$ the extension $\Ext(B_{k} (x,r) , \R')$. This should not cause any confusion.
\end{notation}

\subsection{Tarski-Seidenberg transfer principle}
In some proofs that involve
Morse theory (see for example the proof of Lemma \ref{lem:equivariant_morseB}), where integration of gradient flows is used in an essential way, we
first restrict to the case $\R =\mathbb{R}$. After having proved the result
over $\mathbb{R}$, we use the Tarski-Seidenberg transfer theorem to extend
the result to all real closed fields. We refer the reader to 
{\cite[Chapter 2]{BPRbook2}} for an exposition of the Tarski-Seidenberg transfer
principle.

\subsection{Mayer-Vietoris inequalities}
We will need the following inequalities. They
are consequences of  Mayer-Vietoris exact sequence.

Let $S_{1} , \ldots ,S_{s} \subset \R^{k}$, $s \ge 1$, be closed
semi-algebraic sets of $\R^{k}$, contained in a closed semi-algebraic set $T$.
For $1 \leq t \leq s$, we denote
\begin{eqnarray*}
S_{\le t} &=& \bigcap_{1 \leq j \leq t} S_{j}, \\
S^{\le t} &=& \bigcup_{1 \leq j \leq t} S_{j}.
\end{eqnarray*}
Also, for $J \subset \{1, \ldots ,s\}$, $J \neq \emptyset$, we denote
\begin{eqnarray*}
S_{J} &=& \bigcap_{j \in J} S_{j}, \\
S^{J} &=& \bigcup_{j \in J} S_{j}.
\end{eqnarray*}
Finally, we denote
\begin{eqnarray*}
S^{\emptyset} &=& T.
\end{eqnarray*}

\begin{proposition}
  \label{7:prop:prop1}
   \begin{enumerate}[A.]
    \item 
    \label{itemlabel:7:prop:prop1:1}
    For $i \geq 0$,
    
     \begin{equation}
     \label{7:eqn:prop1:1}
      b_{i} (S^{\le s} ,\F) \leq \sum_{j=1}^{i+1}
      \sum_{\substack{
        J \subset \{ 1, \ldots ,s \}\\
        \card (J) =j
        }}
       b_{i-j+1} (S_{J} ,\F) .
    \end{equation}
    
    \item 
    \label{itemlabel:7:prop:prop1:2}
    For $0 \le i \le k$,
    
     \begin{equation}
      \label{7:eqn:prop1} 
      b_{i} (S_{\le s} ,\F) \leq \sum_{j=1}^{k-i}
      \sum_{\substack{
        J \subset \{ 1, \ldots ,s \}\\
        \card (J) =j
        }}
        b_{i+j-1} (S^{J} ,\F) + \binom{s}{k-i} b_{k}
      (S^{\emptyset} ,\F) .
    \end{equation}
  \end{enumerate}
  \end{proposition}

\begin{proof} See 
\cite[Proposition 7.33]{BPRbook2}.
\end{proof}

We also record 
a special case of Part \eqref{itemlabel:7:prop:prop1:1} of Proposition \ref{7:prop:prop1} for future use.
If $s=2$, then inequality \eqref{7:eqn:prop1:1} gives

\begin{eqnarray}
\label{eqn:MV2}
  b_{i} (S_{1} \cup S_{2} ,\F) & \leq & b_{i} (S_{1}
    ,\F) +b_{i} (S_{2} ,\F)  +b_{i-1} (S_{1} \cap S_{2}
    ,\F). 
\end{eqnarray}

\section{Equivariant deformation\label{sec:deformation}}

In this section we define and prove properties of certain equivariant
deformations of symmetric real algebraic varieties that will be a key
ingredient in the proofs of the main theorems. These are adapted from the
non-equivariant case (see for example {\cite[\S 12.6]{BPRbook2}}), but keeping
everything equivariant requires additional effort.

\begin{notation}
  \label{not:def}For any $P \in \R [ X_{1} , \ldots ,X_{k} ]$ we denote
  \[ \Def (P, \zeta ,d) = P -  \zeta   \left(1+ \sum_{i=1}^{k} X_{i}^{d}
     \right) , \]

  where $\zeta$ is a new variable.
\end{notation}

Notice that if $P$ is symmetric in $X_{1} , \ldots ,X_{k}$, so is $\Def (P,
\zeta ,d)$.

\begin{proposition}
  \label{prop:alg-to-semialg}
  Let $\mathbf{k}= (k_{1} , \ldots ,k_{\omega}) \in \Z_{>0}^{\omega}$, with $k= \sum_{i=1}^{\omega} k_{i}$, 
  and $P \in \R [ \X^{(1)} , \ldots ,\X^{(\omega)} ]^{\mathfrak{S}_\kk}$, where each $\X^{(i)}$ is a
  block of $k_{i}$ variables, and such that $P$ is non-negative.
  Suppose also that $V = \ZZ(P, \R^{k})$ is bounded. The variety
  $\Ext(V, \R \langle \zeta \rangle^{k})$ is
  is semi-algebraically homotopy equivalent to the (symmetric) semi-algebraic subset $S$ of $\R
  \langle \zeta \rangle^{k}$
 consisting of the union of the semi-algebraically connected components of the semi-algebraic set  
  defined by the inequality 
  $\Def (P, \zeta ,d) \leq 0$
which are bounded over $\R$.
  Moreover, 
  $\phi_{\mathbf{k}} (\Ext(V, \R \langle \zeta \rangle^{k}))$
  is semi-algebraically homotopy  equivalent to $\phi_{\mathbf{k}} (S)$. 
\end{proposition}

\begin{proof} 
Let $V \subset B_k(\mathbf{0},R)$ for some $R \in \R, R>0$.
Let for $t \in \R$, $S_t \subset \R^k$ denote the set defined by 
\[
S_t = \{\x= (x_1,\ldots,x_k) \in B_k(\mathbf{0},2R) \;\mid\; P(\x) - t\sum_{i=1}^k x_i^d
  \leq 0\}.
\]

 Then, for all $0< t < t'$, $S_{t} \subset S_{t'}$.  Moreover, $V = \lim_\zeta S$ (cf. Notation \ref{not:lim}). 
 It then follows from \cite[Lemma 17.17]{BPRbook2} that
  $\Ext(V, \R \langle \zeta \rangle^{k})$ is semi-algebraically homotopy equivalent to $S$. 
  
The proof  that
  $\phi_{\mathbf{k}} (\Ext(V, \R \langle \zeta \rangle^{k}))$
  is semi-algebraically homotopy  equivalent to $\phi_{\mathbf{k}} (S)$ is similar and omitted.
\end{proof}

\begin{lemma}
\label{lem:critical}
Let $Q \in \R[X_1,\ldots,X_k]$, and $F = e_1(X_1,\ldots,X_k) = \sum_{i=1}^k X_i$. Then, the critical points of $F$ restricted to
$V = \ZZ(Q,\R^k)$ are defined by the following set of polynomial equations:
\begin{eqnarray}
  Q & = & 0, \nonumber\\
  \dfrac{\partial Q}{\partial X_{1}} -
  \dfrac{\partial Q}{\partial X_{2}} & = & 0, 
  \label{eqn:lem:critical}
   \\
  \vdots & \vdots & \vdots \nonumber \\
  \dfrac{\partial Q}{\partial X_{1}} - \dfrac{\partial Q}{\partial X_{k}} & = & 0. \nonumber
\end{eqnarray}
 \end{lemma}

\begin{proof}
Let $\mathbf{f}_1,\ldots,\mathbf{f}_k$ be the standard basis of $\R^k$ with coordinates
$X_1,\ldots,X_k$. Let $\mathbf{f}_1',\ldots, \mathbf{f}_k'$ be a new basis defined by
\begin{eqnarray*}
\mathbf{f}_1' &=& \sum_{i=1}^k \mathbf{f}_i, \\
\mathbf{f}_2'  &=& \mathbf{f}_1 - \mathbf{f}_2, \\
\vdots &\vdots& \vdots \\
\mathbf{f}_k'  &=& \mathbf{f}_1 - \mathbf{f}_k.
\end{eqnarray*}

Notice that, $\mathbf{f}_1'$ is orthogonal to $\spanof(\mathbf{f}_2',\ldots,\mathbf{f}_k')$, and thus
$\mathbf{f}_2', \ldots,\mathbf{f}_k'$ is a basis of $W =\spanof(\mathbf{f}_1')^{\perp}$. 
The set of critical points of $F$ restricted to  $V$ is the set of points $\x \in V$ where 
\[
\grad(F)(\x) = \sum_{i=1}^k \frac{\partial Q}{\partial X_i}(\x) \mathbf{f}_i
\]
is orthogonal to $W$, or equivalently where
$\grad(F)(\x)$ is orthogonal to each vector $\mathbf{f}_2',\ldots,\mathbf{f}_k'$, since 
$\mathbf{f}_2',\ldots,\mathbf{f}_k'$ span $W$. Thus, the set of critical points of $F$ restricted to  $V$ is defined by
\eqref{eqn:lem:critical}.
\end{proof}

\begin{proposition}
  \label{prop:non-degenerate}Let $P \in \R [ X_{1} , \ldots ,X_{k} ]  $, and
  $d$ be an even number with $\deg (P) \leq d=p+1$, with $p$ a prime. Let
  $F=e_{1} (X_{1} , \ldots ,X_{k})$ where $e_{1}$ denotes the first
  elementary symmetric polynomial. Let 
  \[
  V_{\zeta} = \ZZ \left(\Def (P,\zeta ,d) , \R \langle \zeta \rangle^{k} \right).
  \] 
  Suppose also that $\gcd (p,k) =1$. Then, the critical points of $F$ restricted to $V_{\zeta}$ are
  finite in number, and each critical point is non-degenerate.
\end{proposition}

\begin{proof} 
Using Lemma \ref{lem:critical} with $Q = \Def (P, \zeta ,d)$, we obtain that
the critical points of $F$ restricted to
$V_{\zeta}$ are contained in the set of solutions in $\PP_{\C\la\zeta\ra}^{k}$ of the
following system of homogeneous equations.
\begin{eqnarray}
  \Def (P, \zeta ,d)^{h} & = & 0, \nonumber\\
  \dfrac{\partial \Def (P, \zeta ,d)^{h}}{\partial X_{1}} -
  \dfrac{\partial \Def (P, \zeta ,d)^{h}}{\partial X_{2}} & = & 0, 
  \label{eqn:special1} \\
  \vdots & \vdots & \vdots \nonumber\\
  \dfrac{\partial \Def (P, \zeta ,d)^{h}}{\partial X_{1}} - \dfrac{\partial
  \Def (P, \zeta ,d)^{h}}{\partial X_{k}} & = & 0. \nonumber
\end{eqnarray}

A critical point $x= (x_{1} , \ldots ,x_{k}) \in \R \langle \zeta
\rangle^{k}$ is non-degenerate if and only if the determinant of the Hessian
matrix, $\Hess (x)$, which is an $(k-1) \times (k-1)$ matrix
defined by
\begin{eqnarray*}
  \Hess (x)_{i,j} & = & (\partial_{1} - \partial_{i}) \circ (
  \partial_{1} - \partial_{j})   \Def (P, \zeta ,d) ,
\end{eqnarray*}
(where $\partial_{i} = \dfrac{\partial}{\partial X_{i}}$) is non-zero. In
particular, being non-degenerate implies that a critical point is isolated.

Let $H(P,\zeta,d)$ be defined by 
\[
H(P,\zeta,d) = \det \left( ((\partial_{1} - \partial_{i}) \circ (
  \partial_{1} - \partial_{j})   \Def(P, \zeta ,d)^h)_{2 \leq i, \leq k} \right).
  \]
  
Thus, in order to prove the proposition, it suffices to prove that at each
solution $\bar{x} = (x_{0} :x_{1} : \cdots :x_{k})  $  of the homogeneous
system \eqref{eqn:special1}, 
$H(P,\zeta,d)(x_0:\cdots:x_k) \neq 0$.

Let $\overline{\Def} (P,S_{0} ,S_{1} ,d)^{h}$ 
(resp. $\overline{H}(P,S_0,S_1,d)$)
be the polynomial
obtained from $\Def (P, \zeta ,d)^{h}$ (resp. $H(P,\zeta,d)$)  by first replacing $\zeta$ by
$S_{1}$ and then homogenizing with respect to $S_{1}$, and consider now the
bi-homogeneous system
\begin{eqnarray}
\nonumber
  \overline{\Def} (P,S_{0} ,S_{1} ,d)^{h} & = & 0,\\
  \label{eqn:special1'}
  \dfrac{\partial \overline{\Def} (P,S_{0} ,S_{1} ,d)^{h}}{\partial X_{1}}
  - \dfrac{\partial \overline{\Def} (P,S_{0} ,S_{1} ,d)^{h}}{\partial X_{2}}
  & = & 0,\\
  \nonumber
  \vdots & \vdots & \vdots\\
  \nonumber
  \dfrac{\partial \overline{\Def} (P,S_{0} ,S_{1} ,d)^{h}}{\partial X_{1}}
  - \dfrac{\partial \overline{\Def} (P,S_{0} ,S_{1} ,d)^{h}}{\partial X_{k}}
  & = & 0.
\end{eqnarray}
The set of solutions $(\bar{s} ; \bar{x}) = ((s_{0} :s_{1}) ; (x_{0}
:x_{1} : \cdots :x_{k})) \in \PP_{\C}^{1} \times \PP_{\C}^{k}$ of the above
bi-homogeneous system at which 
$\overline{H}(P,S_0,S_1,d)(\bar{s};\bar{x}) = 0$
is Zariski closed in $\PP_{\C}^{1} \times \PP_{\C}^{k} ,$ and hence, its
projection, $W$,  to $\PP_{\C}^{1}$ is also Zariski closed, and thus is either finite
or equal to $\PP_{\C}^{1}$. 

Note that $\PP_\C^1 \setminus W$, is precisely the set of points
$\bar{s} = (s_0:s_1) \in \PP_\C^1$, such that the polynomial $\overline{H}(P,S_0,S_1,d)(\bar{s};\cdot)$
does not vanish at any point satisfying the set of equations \eqref{eqn:special1'} with
$S_0=s_0,S_1=s_1$.
 
\noindent{\bf Claim:} $(0:1) \not\in W$, and therefore $W$ is finite.
Before we prove this claim below we finish the proof proposition based on this claim.
Since $W$ is finite, its complement, $\PP_\C^1 \setminus W$, contains an open
interval to the right of $0$ of the affine real line, and hence contains the
infinitesimal $\zeta$ after extending the field to $\R \langle \zeta \rangle$.
This implies that for every affine solution
$\bar{x} = (1:x_1:\cdots:x_k)$ of \eqref{eqn:special1}, 
\[
\overline{H}(P,S_0,S_1,d)((1:\zeta);1:x_1:\cdots:x_k) = \Hess(x_1,\ldots,x_k) \neq 0,
\]
and hence every critical point of $F$ restricted to $V$ is non-degenerate proving the proposition.

We now prove the claim that $(0:1) \not\in W$. We obtain after substituting 
$S_{0} =0,S_{1} =1$
in 
\eqref{eqn:special1'}
the following system
\begin{eqnarray}
  X_{0}^{d} + \sum_{i=1}^{k} X_{i}^{d} 
  & = & 0,
  \nonumber\\
  X_{1}^{d-1} -X_{2}^{d-1} & = & 0, \label{eqn:special2} \\
  \vdots & \vdots & \vdots \nonumber\\
  X_{1}^{d-1} -X_{k}^{d-1} & = & 0. \nonumber
\end{eqnarray}
Notice that for any solution $x = (x_{0} :x_{1} : \cdots :x_{k})$ to the
system of equations \eqref{eqn:special2} we must have that for $i=2, \ldots
,k$,
\begin{eqnarray}
  x_{i} & = & \omega_{i} x_{1,}  \label{eqn:xi}
\end{eqnarray}
where each $\omega_{i}$ is a $p$-th root of unity (note that $p=d-1$). 

Now,
$$\displaylines{
  \overline{H}(P,S_0,S_1,d)((0:1);\bar{x}) = \cr
  \left(\begin{array}{cccc}
    x_{1}^{d-2} +x_{2}^{d-2} & x_{1}^{d-2} & \cdots & x_{1}^{d-2}\\
    x_{1}^{d-2} & x_{1}^{d-2} +x_{3}^{d-2} & \cdots & x_{1}^{d-2}\\
    \vdots & \vdots & \ddots & \vdots\\
    x_{1}^{d-2} & x_{1}^{d-2} & \cdots & x_{1}^{d-2} +x_{k}^{d-2}
  \end{array}\right) .
}
$$
Noting that $x_{1} \neq 0$, and substituting for the various $x_{i} ,2 \leq i
\leq k$, using \eqref{eqn:xi}  we get that
\begin{eqnarray*}
   \overline{H}(P,S_0,S_1,d)((0:1);\bar{x})
  & = & x_{1}^{(d-2) (k-1)}
  \left(\prod_{i=2}^{k} \omega_{i}^{d-2} \right) \left(1+ \sum_{i=2}^{k}
  \omega_{i}^{d-2} \right) .
\end{eqnarray*}

Since $p$ is prime, the only integral relations between the $p$-th roots of
unity are integer multiples of the relation
\[ 1+ \omega + \cdots + \omega^{p-1} =0   , \]
where $\omega$ is a primitive $p$-th root of unity. Since, $p$ does not divide
$k$ by hypothesis, it follows that
\begin{eqnarray*}
  1+ \sum_{i=2}^{k} \omega_{i}^{d-2} & \neq & 0
\end{eqnarray*}
for any choice of the roots $\omega_{i}$. Hence, 
 $\overline{H}(P,S_0,S_1,d)((0:1);\bar{x}) \neq 0$.
This finishes the
proof.\end{proof}

\begin{lemma}
  \label{lem:half-degree}Let $\mathbf{k}= (k_{1} , \ldots ,k_{\omega}) \in
  \Z_{>0}^{\omega}$, with $k= \sum_{i=1}^{\omega} k_{i}$, and let $Q
  \in \R [ \X^{(1)} , \ldots
  ,\X^{(\omega)} ]$, where each
  $\X^{(i)}$ is a block of $k_{i}$ variables, and
  such that $Q$ is non-negative over $\R^{k}$, and symmetric in each of the
  blocks $\X^{(i)}$. Let $\deg   (Q) \leq d$, $d$
  an even number, and suppose that $\ZZ \left(Q, \R^{k} \right)$ is a finite
  set of points. Then, for each $(\x^{(1)} , \ldots
  ,\x^{(\omega)}) \in \ZZ \left(Q, \R^{k} \right)$, we have that
  for each $i,1 \leq i \leq \omega$, $\card
  \left(\bigcup_{1 \leq j \leq k_{i}} \{ x^{(i)}_{j} \} \right) \leq d/2$
  (where $\x^{(i)} = (x^{(i)}_{1} , \ldots ,x^{(i)}_{k_{i}}
 )$).
\end{lemma}

\begin{proof} We assume without loss of generality that 
$i=\omega$, and let $\Y$ denote the variables $(\X^{(1)} ,
\ldots ,\X^{(\omega -1)})$. First notice that there exists
polynomials
$G_{0} ,G_{d/2+1} , \ldots ,G_{d} \in \R[\Y,Z_1,\ldots,Z_{d/2}]$
such that
\begin{eqnarray}
  Q & = & G_{0} (\Y,e_{1} , \ldots ,e_{d/2}) + \sum_{i=d/2+1}^{d}
  G_{i} (\Y, e_{1} , \ldots , 
  e_{d/2})
  e_{i} 
\end{eqnarray}
where $e_{i} (\X^{(\omega)})$ is the $i$-th elementary symmetric
polynomial in $\X^{(\omega)}$.

Let 
$\x= (\mathbf{y},\x^{(\omega)}) \in \ZZ \left(Q, \R^{k}\right)$ be
such that 
\[
\ell := \ell (\x^{(\omega)}) = \card(\bigcup_{1 \leq j \leq k_{\omega}} \{ x_{j}^{(\omega)} \}),
\]
where $\x^{(\omega)} = (x^{(\omega)}_{1} , \ldots ,x^{(\omega)}_{k_{\omega}})$, 
is maximum amongst all the points $\X$ belonging to the  finite set $\ZZ(Q, \R^{k})$. 
The proof of the lemma is by contradiction. Suppose that
$\ell >d/2$. There are two cases to consider
-- namely, the case when $\ell = k_\omega$, and the case $d/2 < \ell < k_\omega$. We treat each one separately below.
\\

The case  $\ell =k$:   Since the roots of a univariate polynomial depend
  continuously on the coefficients we have that there is a $\eps_{0} > 0$, such that for every 
  $\xi = (\xi_{0} , \ldots , \xi_{k_{\omega}-1}) \in \R^{k_{\omega}}$, with $| \xi | < \eps_{0}$, the polynomial
  \begin{eqnarray*}
    f_{\xi} & = & \sum_{j=0}^{k_{\omega} -1} (-1)^{k_{\omega} -j} (e_{k-j}
    (x) + \xi_{j}) T^{j}  +T^{k_{\omega}}\\
    &  & 
  \end{eqnarray*}
  also has $k_{\omega}$ distinct real roots
  (since having all roots real is an open condition on the space of real monic polynomials of a given degree). 
  Considering these $k_{\omega}$  real roots of $f_{\xi}$ as the $k_{\omega}$ components of a point 
  $\theta(\xi) \in \R^{k_{\omega}}$ we get a 
  differentiable map 
  \begin{equation*}
  \theta : B_{k_\omega}(\mathbf{0},\eps_0)  \rightarrow \R^{k_{\omega}}.
  \end{equation*}

   Using the fact that all the roots of $f_\xi$ are distinct for $\xi \in B_{k_\omega}(\mathbf{0},\eps_0)$, it is a simple exercise to check
   that the Jacobian of the map $\theta$ has non-vanishing determinant at all $\xi \in B_{k_\omega}(\mathbf{0},\eps_0)$, and hence $\theta$
  is a diffeomorphism on to its image (by the inverse function theorem).
  
 Clearly the set $U =  \{ \mathbf{y} \} \times \theta (V_{\xi})$ where
 \[ 
 V_{\xi} = B_{k_\omega}(\mathbf{0},\eps_0)  \cap \{ \xi  \mid \xi_{d/2+1} = \cdots = \xi_{k_\omega} =0 \} 
 \]
  contains 
  $\x$.
  
  Notice that since $d/2<k_{\omega}$, dimension of $V_{\xi}$
  and hence that of $U$ is at least one.

 Now if $G_{i} (\mathbf{y},e_{1} (\x^{(\omega)}) , \ldots ,e_{d/2} (\x^{(\omega)})) =0$ 
 for all $i,d/2+1 \leq i \leq k_{\omega}$, then 
for all $\x' = (\y,\zb') \in U$,
\begin{eqnarray*}
Q(\y,\zb') &=& 
 G_{0} (\y,e_{1}(\zb') , \ldots ,e_{d/2}(\zb')) +\\
 && \sum_{i=d/2+1}^{d}
  G_{i} (\y, e_{1}(\zb') , \ldots , 
  e_{d/2}(\zb'))
  e_{i}(\zb')\\
  &=&
  G_{0} (\y,e_{1}(\x^{(\omega)})+\xi_{k_\omega} , \ldots ,e_{d/2}(\x^{(\omega)})+\xi_{d/2}) + \\
  &&\sum_{i=d/2+1}^{d}
  G_{i} (\y, e_{1}(\x^{(\omega)})+
 \xi_{k_\omega} , \ldots , 
  e_{d/2}(\x^{(\omega)})+ \xi_{d/2})
  (e_{i}(\zb')+\xi_{k_\omega -i})\\
  &=&
  G_{0} (\y,e_{1}(\x^{(\omega)}) , \ldots ,e_{d/2}(\x^{(\omega)}) ) + \\
  &&\sum_{i=d/2+1}^{d}
  G_{i} (\y, e_{1}(\x^{(\omega)})+ , \ldots , 
  e_{d/2}(\x^{(\omega)}))
  (e_{i}(\zb')+\xi_{k_\omega -i}) \\
  &=&
  G_{0} (\y,e_{1}(\x^{(\omega)}) , \ldots ,e_{d/2}(\x^{(\omega)})) \\
  &=& 0,
  \end{eqnarray*}
and hence
 $U \subset \ZZ (Q, \R^{k})$ which contradicts the assumption that 
 $\ZZ(Q, \R^{k})$ is a finite set of points.
 
  Otherwise, if 
  \[
  G_{i} (
  \mathbf{y},e_{1} (\x^{(\omega)}) , \ldots ,e_{d/2} (
  \x^{(\omega)}))   \neq 0
  \] 
  for some $i,d/2+1 \leq i \leq  k_{\omega}$,
   then supposing that 
   \[
   G_{i} (\mathbf{y},e_{1} (\x^{(
  \omega)}) , \ldots ,e_{d/2} (\x^{(\omega)}))  > 0
  \]
  (respectively $G_{i} (\mathbf{y},e_{1} (\x^{(\omega)}) ,
  \ldots ,e_{d/2} (\x^{(\omega)}))  <0$), 
  
  \[
  Q(\mathbf{y}, \theta ((\underbrace{0,\ldots,0}_{k_\omega-i-1},\eps,\underbrace{0\ldots,0}_{i}) )) <0
  \]
  for all $\eps,   -\eps_0< \eps < 0$ (respectively $0 < \eps <\eps_0$).
  This contradicts the fact that $Q$ is non-negative everywhere. \\
  
  The case  $d/2< \ell <k_{\omega}$:   In this case by Proposition 3.2 in
  \cite{Riener} there exists a univariate polynomial
  \begin{eqnarray*}
    g & = & \sum_{j=0}^{k_{\omega} - \ell}   (-1)^{k_{\omega} -j} g_{j}
    T^{j}
  \end{eqnarray*}
  having the following property.
  Let,
  \begin{eqnarray*}
  f  &=& \prod_{i=1}^{k_\omega} (T - x^{(\omega)}_i),
  \end{eqnarray*}
  and 
  \begin{eqnarray*}
   h_{\eps} & = & f + \eps g.
  \end{eqnarray*}
  Then, there exists $\eps_0 > 0$, such that for all $\eps \neq 0$, with $| \eps| < \eps_0$, 
  $h_\eps$ 
  is monic,
  has all its roots real, and moreover has at least $\ell +1$ distinct roots.

  Considering now the $k$ real roots of $h_{\eps}$ as the $k$ components of a
  point 
  in 
  $\R^{k_{\omega}}$, we obtain a
  continuous (non-constant) semi-algebraic curve 
  \[
  \gamma : (- \eps_{0}, \eps_{0}) \rightarrow \R^{k_{\omega}}.
  \]
  Note that
  the curve is non-constant,
  since for all $\eps \in (- \eps_{0} ,\eps_{0})$ with $\eps \neq 0$, $\gamma(\eps)$ has
  strictly more distinct components than $\x^{(\omega)}$, and hence
  $\gamma(\eps) \neq \x^{(\omega)}$.

  It follows that
  for each $\eps \in (- \eps_{0} , \eps_{0})$
  \begin{eqnarray*}
  Q (\mathbf{y}, \gamma(\eps)) & =  &
    G_{0} (\y,e_{1}(\x^{(\omega)}) , \ldots ,e_{d/2}(\x^{(\omega)})) + \\
  &&\sum_{i=d/2+1}^{\ell}
  G_{i} (\y, e_{1}(\x^{(\omega)})+ , \ldots , 
  e_{d/2}(\x^{(\omega)}))
  e_{i}(\x^{(\omega)}) + \\ 
  &&\sum_{i=\ell+1}^{k_\omega}
  G_{i} (\y, e_{1}(\x^{(\omega)})+ , \ldots , 
  e_{d/2}(\x^{(\omega)}))
  (e_{i}(\x^{(\omega)})+\eps g_{k_\omega-i}) \\
  &=&
  \eps (
    \sum_{j=0}^{k_{\omega} - \ell} g_{j} G_{k-j} (\mathbf{y},e_{1} (
    \x^{(\omega)}) , \ldots ,e_{d/2} (\x^{(\omega)}))
    ) .
  \end{eqnarray*}
  There are again two cases. If 
  \[
  \sum_{j=0}^{k_{\omega} - \ell} g_{j}
  G_{k_{\omega} -j} (\mathbf{y},e_{1} (\x^{(\omega)}) , \ldots
  ,e_{d/2} (\x^{(\omega)})) =0,
  \] 
  then $Q$ vanishes on $\gamma
  \left(\left(- \eps_{0} , \eps_{0} \right) \right)$, which contradicts the
  hypothesis that  $\ZZ \left(Q, \R^{k} \right)$ is a finite
  set of points.
  Otherwise, if
  \[
  \sum_{j=0}^{k_{\omega} - \ell} g_{j} G_{k_{\omega} -j} (\mathbf{y},e_{1} (
  \x^{(\omega)}) , \ldots ,e_{d/2} (\x^{(\omega)}))
  \neq 0,
  \]
 then 
 \[
 Q \left(\mathbf{y}, \gamma \left(\eps \right) \right)
  \cdot Q \left(\mathbf{y}, \gamma \left(- \eps \right) \right) <0,
  \] 
  for  every $\eps \in \left(- \eps_{0} , \eps_{0} \right) , \eps \neq 0$, and
  this contradicts the hypothesis that $Q$ is non-negative everywhere.
\end{proof}

Before proving the next proposition we introduce a notation.

\begin{notation}
  For any pair $(\mathbf{k},\boldsymbol{\ell})$, where $\mathbf{k}= (k_{1} ,
  \ldots ,k_{\omega}) \in \Z_{>0}^{\omega}$, $k= \sum_{i=1}^{\omega}
  k_{i}$, and $\boldsymbol{\ell}= (\ell_{1} , \ldots , \ell_{\omega})$, with $1
  \leq \ell_{i} \leq k_{i}$, we denote by $A_{\mathbf{k}}^{\boldsymbol{\ell}}$ the
  subset of $\R^{k}$ defined by
  \begin{eqnarray*}
    A^{\boldsymbol{\ell}}_{\mathbf{k}} & = & \left\{ x= (x^{(1)} , \ldots x^{(
    \omega)}) \mid  \card \left(
    \bigcup_{j=1}^{k_{i}} \{ x_{j}^{(i)} \} \right) = \ell_{i} \right\} .
  \end{eqnarray*}
\end{notation}

\begin{proposition}
  \label{prop:half-degree}Let $\mathbf{k}= (k_{1} , \ldots ,k_{\omega})
  \in \Z_{>0}^{\omega}$, with $k= \sum_{i=1}^{\omega} k_{i}$, and $P
  \in \R [ \X^{(1)} , \ldots
  ,\X^{(\omega)} ]$, where each
  $\X^{(i)}$ is a block of $k_{i}$ variables, such
  that $P$ is non-negative and symmetric in each block of variable
  $\X^{(i)}$ and $\deg (P) \leq d$. Let $(X_{1}
  , \ldots ,X_{k})$ denote the set of variables $(
  \X^{(1)} , \ldots ,\X^{(
  \omega)})$ and let $F=e_{1}^{(k)} (X_{1} , \ldots ,X_{k})$. Suppose
  that the critical points of $F$ restricted to $V= \ZZ \left(P, \R^{k}
  \right)$ are isolated. Then, each critical point of $F$ restricted to $V$ is
  contained in $A^{\boldsymbol{\ell}}_{\mathbf{k}}$ for some $\boldsymbol{\ell}= (\ell_{1}
  , \ldots , \ell_{\omega})$ with each $\ell_{i} \leq d$.
\end{proposition}

\begin{proof} 
Let  $\X = (X_{1} , \ldots ,X_{k})$ denote the
variables $(\X^{(1)} , \ldots ,\X^{(\omega)})$. 
Let 
\[
\x= (x_{1} , \ldots ,x_{k})
\] 
be a critical point of $F$ restricted to $V$. 
Then, $\x$ is an isolated zero (in fact a local minima) of the polynomial
\begin{eqnarray*}
  Q & = & P^{2} + \sum_{i,j=1}^{k} \left(\frac{\partial  P}{\partial X_{i}} -
  \frac{\partial P}{\partial X_{j}} \right)^{2} .
\end{eqnarray*}
Notice that $Q$ is symmetric in each block of variables $\X^{(1)}
, \ldots ,\X^{(\omega)}$ and $\deg (Q) \leq 2d$. Now apply Lemma
\ref{lem:half-degree}.

\end{proof}

Before proceeding further we need some more notation.

\begin{notation}
\label{noy:L-pi}
  Let $\boldpi \in
  \boldPi_{\mathbf{k}}$ where $\mathbf{k}= (k_{1} ,
  \ldots ,k_{\omega}) \in \Z_{>0}^{\omega}$, with $k=
  \sum_{i=1}^{\omega} k_{i}$. 
  
  For $1 \leq i  \leq \omega$, and $1 \leq j \leq
  \length (\pi^{(i)})$, let $L_{\pi^{(i
 )}_{j}} \subset \R^{k}$ be defined by the equations
  \begin{eqnarray*}
    X^{(i)}_{\pi_{1}^{(i)} + \cdots + \pi_{j-1}^{(i)} +1} & = \cdots = &
    X^{(i)}_{\pi_{1}^{(i)} + \cdots + \pi_{j}^{(i)}},
  \end{eqnarray*}
and let
  \begin{eqnarray*}
    L_{\boldpi} & = & \bigcap_{1 \leq i \leq \omega}
    \bigcap_{1 \leq j \leq \length (\pi^{(i)})}
    L_{\pi^{(i)}_{j}} . 
  \end{eqnarray*}
  \end{notation}

\begin{notation}
\label{not:pi-of-x}
For $\x \in \R^k$ or $\C^k$, let $G_\x$ be the isotropy subgroup of $\x$ with respect to the action of
$\mathfrak{S}_k$ on $\R^k$ or $\C^k$ permuting coordinates. Then, it is easy to verify that 
\[
G_\x \cong \mathfrak{S}_{\ell_1} \times \cdots \times \mathfrak{S}_{\ell_m},
\]
where $k \geq \ell_1 \geq \ell_2 \geq \cdots \geq \ell_m > 0, \sum_i \ell_i = k$,
and $\ell_1,\ldots,\ell_m$  are the  
cardinalities of the sets 
\[
\{i \mid 1 \leq i \leq k, x_i = x\}, x \in \bigcup_{i=1}^k \{x_i \}
\] 
in non-decreasing order.
We denote by $\pi(\x)$ the partition $(\ell_1,\ldots,\ell_m) \in \Pi_k$.

More generally, for $\mathbf{k}= (k_{1} , \ldots ,k_{\omega})
  \in \Z_{>0}^{\omega}$, with $k= \sum_{i=1}^{\omega} k_{i}$, and 
  $\x = (\x^{(1)},\ldots,\x^{(\omega)}) \in \R^k$, where each $\x^{(i)} \in \R^{k_i}$, we denote 
  \[
  \boldpi(\x) = (\pi(\x^{(1)}),\ldots,\pi(\x^{(\omega)})) \in \boldPi_\kk.
  \]
\end{notation}

\begin{proposition}
  \label{prop:betti_bound}
  Let $\mathbf{k}= (k_{1} , \ldots ,k_{\omega})
  \in \Z_{>0}^{\omega}$, with $k= \sum_{i=1}^{\omega} k_{i}$, and let
  $S \subset \R^{k}$ be a bounded symmetric basic closed semi-algebraic set
  defined by $P \leq 0$, where $P \in \R [ \X^{(1
 )} , \ldots ,\X^{(\omega)} ]$ is symmetric in
  each block of $k_{i}$ variables $\X^{(i)}$, and
  such that $W= \ZZ \left(P, \R^{k} \right)$ is non-singular and bounded.
  Suppose that $F=e_{1} (\X^{(1)} , \ldots
  ,\X^{(\omega)})$ restricted to $W$ has a finite
  number of critical points, all of which are non-degenerate. Let $C$ denote
  the finite set of critical points of $F$ restricted to $W$. Then, for any
  field of coefficients $\F$,
  \begin{eqnarray*}
    b (\phi_{\mathbf{k}} (S) ,\F) & \leq & \frac{1}{2}
    \card (\phi_{\mathbf{k}} (C)) .
  \end{eqnarray*}
  Moreover,
  \begin{eqnarray}
    b_{i} (\phi_{\mathbf{k}} (S) ,\F) & = & 0 
    \label{eqn:prop-main}
  \end{eqnarray}
  for
  \begin{eqnarray*}
    i & \geq & \max_{\x \in C}   (\length (
    \boldpi (\x))).
  \end{eqnarray*}
\end{proposition}

For the proof of Proposition \ref{prop:betti_bound} we will need the following
proposition and lemmas.

\begin{proposition}
  \label{prop:orthogonal}Let $L \subset \R^{k}$ be the subspace defined by
  $\sum_{i} X_{i} =0$, and $\boldpi= (\pi^{(1)} ,
  \ldots , \pi^{(\omega)}) \in \boldPi_{\mathbf{k}}$.
  Let for each i, $1 \leq i \leq \omega$, $\pi^{(i)} = (\pi^{(i)}_{1} ,
  \ldots , \pi^{(i)}_{\ell_{i}})$, and for each $j,1 \leq j \leq \ell_{i}
  ,$ let $L^{(i)}_{j}$ denote the subspace $L \cap L_{\pi^{(i)}_{j}}$ of
  $L$, and $M^{(i)}_{j}$ the orthogonal complement of $L^{(i)}_{j}$ in
  $L$. Let $L_{\fixed} =L \cap
  L_{\boldpi}$, $L'_{\fixed}
  \subset L_{\fixed}$ any subspace of
  $L_{\fixed}$, and $I  \subset \{ (i,j) \mid 1
  \leq i \leq \omega ,1 \leq j \leq \ell_{i} \}$. Then the following hold.
  \begin{enumerate}[A.]
    \item
    \label{itemlabel:prop:orthogonal:1}
     The dimension of $L_{\fixed}$ is equal to $ 
    \sum_{i=1}^{\omega} \ell_{i} -1= \length (\boldpi) -1$.
    \item 
    \label{itemlabel:prop:orthogonal:2}
    The product over $i \in [ 1, \omega ]$ of the subgroups
    $\mathfrak{S}_{\pi^{(i)}_{1}} \times \mathfrak{S}_{\pi^{(i)}_{2}}
    \times \cdots \times \mathfrak{S}_{\pi^{(i)}_{\ell_{i}}}$ acts trivially
    on $L_{\fixed}$.
    
    \item
    \label{itemlabel:prop:orthogonal:3}
     For each $i,j,1 \leq i \leq \omega ,1 \leq j \leq \ell_{i}$, $M^{(i
   )}_{j}$ is an irreducible representation of $\mathfrak{S}_{\pi^{(i
   )}_{j}}$, and the action of $\mathfrak{S}_{\pi^{(i')}_{j'}}$ on $M^{(
    i)}_{j}$ is trivial if $(i,j) \neq (i' ,j')$.
    
    \item 
    \label{itemlabel:prop:orthogonal:4}
    There is a direct decomposition $L=L_{\fixed} \oplus \left(
    \bigoplus_{1 \leq i \leq \omega ,1 \leq j \leq \ell_{i}} M^{(i)}_{j}
    \right)$.
    
    \item 
    \label{itemlabel:prop:orthogonal:5}
    Let $\mathbf{D}$ denote the unit disc in the subspace
    $L_{\fixed}' \oplus \left(\bigoplus_{(i,j) \in I} M^{(i)}_{j}
    \right)$. Then, the space of orbits of the pair $(\mathbf{D}, \partial
    \mathbf{D})$ under the action of $\mathfrak{S}_{\mathbf{k}}$ is homotopy
    equivalent to $(\ast , \ast)$ if $I \neq \emptyset$. Otherwise, the
    space of orbits of the pair $(\mathbf{D}, \partial \mathbf{D})$ under
    the action of $\mathfrak{S}_{\mathbf{k}}$ is homeomorphic to $(
    \mathbf{D}, \partial \mathbf{D})$.
  \end{enumerate}
\end{proposition}

\begin{proof}[Proof of Proposition \ref{prop:orthogonal}] 
From the definition of $L_{\fixed}$ it is
clear that
\begin{eqnarray*}
  \dim  L_{\fixed} & = &   
  \left(k-1 -  \left(\sum_{\substack{1 \leq i \leq
  \omega, \\1 \leq j \leq \ell_{i}}} (\pi^{(i)}_{j} -1) \right) \right)\\
  & = & \sum_{i=1}^{\omega} \ell_{i} -1\\
  & = & \length (\boldpi) -1,
\end{eqnarray*}
noting that for each $i,1 \leq i \leq \omega ,$ $\sum_{1 \leq j \leq \ell_{i}}
\pi^{(i)}_{j} =k_{i}$, and $\sum_{i=1}^{\omega} k_{i} =k$.
This proves Part \eqref{itemlabel:prop:orthogonal:1}.

Parts 
\eqref{itemlabel:prop:orthogonal:2} and \eqref{itemlabel:prop:orthogonal:3}
are now clear from the definition of the subspaces
$L_{\fixed}$ and the subspaces $M^{(i)}_{j}$. 

In order to prove Part
\eqref{itemlabel:prop:orthogonal:4}
notice that each $M^{(i)}_{j}$ is orthogonal complement of $L^{(i
)}_{j}$ in $L$, 
$\dim  L^{(i)}_{j} + \dim  M^{(i)}_{j} =k-1$. 
Moreover,
$\dim  L^{(i)}_{j} =k-1- (\pi^{(i)}_{j} -1) =k- \pi^{(i)}_{j}$. 
Hence,
$\dim  M^{(i)}_{j} = \pi^{(i)}_{j} -1$. 
Now since 
\[
L_{\fixed} =  \bigcap_{\substack{1 \leq i \leq \omega, \\1 \leq j \leq \ell_{i}}} L^{(i)}_{j},
\] 
it  follows that
$
\sum_{\substack{1 \leq i \leq \omega, \\1 \leq j \leq \ell_{i}}} M^{(i)}_{j}
$
is the orthogonal complement of $L_{\fixed}$ in $L$. 
Hence, 
\[
L= L_{\fixed} \oplus \left(\sum_{\substack{1 \leq i \leq \omega, \\1 \leq j \leq \ell_{i}}} M^{(i)} \right),
\] 
and hence 

\begin{eqnarray*}
\dim   \left(\sum_{\substack{1 \leq i \leq  \omega, \\1 \leq j \leq \ell_{i}}} M^{(i)} \right) &=&  
\dim  L-  \dim   L_{\fixed} \\
&=&  (k-1) -  \left(k-1 -  \sum_{1 \leq i \leq \omega ,1 \leq
j \leq \ell_{i}} (\pi^{(i)}_{j} -1) \right) \\
&=& \sum_{\substack{1 \leq i \leq \omega, \\1 \leq j \leq \ell_{i}}} (\pi^{(i)}_{j} -1) \\
&=& \sum_{\substack{1 \leq i \leq \omega, \\1 \leq j \leq \ell_{i}}} \dim  M^{(i)}_{j}.
\end{eqnarray*}
It follows, that 
\[
\sum_{\substack{1 \leq i \leq \omega, \\1 \leq j \leq \ell_{i}}}  M^{(i)}_{j}   \simeq   \bigoplus_{\substack{1
\leq i \leq \omega, \\1 \leq j \leq \ell_{i}}} M^{(i)}_{j}.
\]

In order to prove Part 
\eqref{itemlabel:prop:orthogonal:5}
first observe that the space of orbits of
$\partial \mathbf{D}$ (respectively $\mathbf{D}$) under the action of
$\mathfrak{S}_{\mathbf{k}}$ is homeomorphic to the quotient
$\partial \mathbf{D}/ \prod_{(i,j) \in I} \mathfrak{S}_{\pi^{(i)}_{j}}$
(respectively $\mathbf{D}/ \prod_{(i,j) \in I} \mathfrak{S}_{\pi^{(i
)}_{j}}$). Moreover, $\partial \mathbf{D}$ is equivariantly homeomorphic to
the topological join of $\partial \mathbf{D}_{\fixed}$ with the various
$\partial \mathbf{D}^{(i)}_{j} , (i,j) \in I$ where
$\mathbf{D}_{\fixed}$ is the unit disc in $L_{\fixed}$, and for
each $(i,j) \in I$, $\mathbf{D}^{(i)}_{j}$ is the unit disc in the
subspace $M^{(i)}_{j}$. The subgroup $\prod_{(i,j) \in I}
\mathfrak{S}_{\pi^{(i)}_{j}}$ acts trivially on $\partial
\mathbf{D}_{\fixed}$, and it follows from 
Part \eqref{itemlabel:prop:orthogonal:3}
of the proposition
that for each $(i,j) \in I$,
\begin{eqnarray*}
  \partial \mathbf{D}^{(i)}_{j} / \prod_{(i,j) \in I}
  \mathfrak{S}_{\pi^{(i)}_{j}} & \simeq_{\tmop{homeo}} & \partial
  \mathbf{D}^{(i)}_{j} /\mathfrak{S}_{\pi^{(i)}_{j}} .
\end{eqnarray*}
Hence, we get that the quotient of the topological join of $\partial
\mathbf{D}_{\fixed}$ with the various $\partial \mathbf{D}^{(i
)}_{j} , (i,j) \in I$ by 
$\prod_{(i,j) \in I} \mathfrak{S}_{\pi^{(i
)}_{j}}$ is homeomorphic to the topological join of $\partial
\mathbf{D}_{\fixed}$ with the various 
$$\partial \mathbf{D}^{(i)}_{j}
/\mathfrak{S}_{\pi^{(i)}_{j}} , (i,j) \in I.
$$

It follows from  {\cite[Theorem 4.1.8]{Reiner}} that each $\partial
\mathbf{D}^{(i)}_{j} /\mathfrak{S}_{\pi^{(i)}_{j}} , (i,j) \in I$ is
homeomorphic to $\mathbf{D}^{(i)}_{j}$ and hence homotopy equivalent to a
point. The quotient of the disc {\tmstrong{$\mathbf{D}$}} by $\prod_{(i,j)
\in I} \mathfrak{S}_{\pi^{(i)}_{j}}$ is clearly contractible. 
This proves both parts of 
\eqref{itemlabel:prop:orthogonal:5}.
\end{proof}

The proof of Proposition \ref{prop:betti_bound} will now follow from the
following two lemmas.
Following the
same notation as in Proposition \ref{prop:betti_bound}, and for any $c \in
\R$, let $S_{\leq c}$ (respectively $S_{=c}$) denote the set $S \cap F^{-1} (
(- \infty ,c ])$ (respectively $S \cap F^{-1} (c)$). Also, let $c_{1} ,
\ldots ,c_{N}$ be the finite set of critical values of $F$ restricted to $W$.

\begin{lemma}
  \label{lem:equivariant_morseA} Then, for $1 \leq i<N$, and for each $c \in
  [ c_{i} ,c_{i+1})$, $\phi_{\mathbf{k}} (S_{\leq c})$ is
  semi-algebraically homotopy equivalent to $\phi_{\mathbf{k}} (S_{\leq
  c_{i}})$.
\end{lemma}

\begin{proof}The lemma is an equivariant version of the standard
Morse Lemma A. It follows from the fact that the gradient flow, which gives a
retraction of $S_{\leq c}$ to $S_{\leq c_{i}}$, is equivariant, and thus
descends to give a retraction of $\phi_{\mathbf{k}} (S_{\leq c})$ to
$\phi_{\mathbf{k}} (S_{\leq c_{i}})$.\end{proof}

We also need the following equivariant version of Morse Lemma B.

Using the same notation as in Proposition \ref{prop:betti_bound}:

\begin{lemma}
  \label{lem:equivariant_morseB}
  Let $G_c^-$ denote a set of representatives of orbits of critical points $\x$  of  $F$ restricted to $W$ with $F(\x) = c$, and 
\begin{eqnarray}
\label{eqn:equivariant_morseB}
    \sum_{1 \leq i \leq k} \dfrac{\partial P}{\partial X_{i}} (\x)  &<& 0.
\end{eqnarray}
  
    Then, for for all small enough $t>0$,
    
    \begin{enumerate}[A.]
     \item 
    \label{itemlabel:lem:equivariant_morseB:1}
    
    \begin{eqnarray}
    \label{eqn:inequality1}
      b (\phi_{\mathbf{k}} (S_{\leq c}) ,\F) & = & b (
      \phi_{\mathbf{k}} (S_{\leq c-t}) ,\F) + \card(G_c^-).
    \end{eqnarray}
    
    \hide{
    \item 
    \label{itemlabel:lem:equivariant_morseB:2}
    Suppose that for each critical point $\x \in W$, with $F (\x) =c$, 
    \[
    \sum_{1 \leq i \leq k} \dfrac{\partial P}{\partial X_{i}} (\x)  <0.
    \]
    Then,  for all small enough $t>0$,
    \begin{eqnarray}
    \label{eqn:inequality2}
      b (\phi_{\mathbf{k}} (S_{\leq c+t}) ,\F) & \leq & b (
      \phi_{\mathbf{k}} (S_{\leq c}) ,\F) +1.
      \end{eqnarray}
   }
    
    \item 
    \label{itemlabel:lem:equivariant_morseB:3}
    Moreover, 
    \begin{eqnarray}
      b_{i} (\phi_{\mathbf{k}} (S_{\leq c}) ,\F) & = & b_{i} (
      \phi_{\mathbf{k}} (S_{\leq c-t}) ,\F)  \label{eqn:inequality3}
    \end{eqnarray}
    for all $i  \geq \max_{\x \in G_c^-}  \length (\boldpi(\x))$.
  \end{enumerate}
\end{lemma}

\begin{proof} We first prove the proposition for $\R =
\mathbb{R}$. 
We will also assume that the function $F$ takes distinct values on the distinct orbits
  of the critical points of $F$ restricted to $W$ for ease of exposition of the proof.  Since
  the topological changes at the critical values are local near the critical points which are assumed to be isolated, the general case follows easily using a standard partition of unity argument. 
 Also, note that the value of $\sign(\sum_{1 \leq i \leq k} \dfrac{\partial P}{\partial X_{i}} (\x) )$
 are equal for all critical points $\x$ belonging to one orbit.\\

 \noindent{Proof of Part \eqref{itemlabel:lem:equivariant_morseB:1}:}  
 If 
 \[
 \sum_{1 \leq i \leq k} \dfrac{\partial P}{\partial X_{i}} (\x) 
  >0,
  \]
  then $S_{\leq c}$ retracts
  $\mathfrak{S}_{\mathbf{k}}$-equivariantly to a space $S_{\leq c-t} \cup_{B}
  A$ where the pair $(A,B) = \coprod_{\x} ( A_{\x} ,B_{\x})$,  and where the
  disjoint union is taken over the set critical points $\x$ with $F (\x) =c$,
  and each pair $(A_{\x} ,B_{\x})$ is homeomorphic to the pair $(
  \mathbf{D}^{i} \times [ 0,1 ] , \partial \mathbf{D}^{i} \times [ 0,1 ] \cup
  \mathbf{D}^{i} \times \{ 1 \})$,
  where $i$ is the dimension of the negative eigenspace of the Hessian  of the
  function $e_1^{(k)}$ restricted to $W$ at $\x$.
  This follows from the basic Morse theory
  (see {\cite[Proposition 7.19]{BPRbook2}}). Since the pair $(
  \mathbf{D}^{i} \times [ 0,1 ] , \partial \mathbf{D}^{i} \times [ 0,1 ] \cup
  \mathbf{D}^{i} \times \{ 1 \})$ is homotopy equivalent to $(\ast , \ast
 )$, $S_{\leq c-t}$ is homotopy equivalent to $S_{\leq c}$, and it follows
  that $\phi_{\mathbf{k}} (S_{\leq c-t})$ is homotopy equivalent to
  $\phi_{\mathbf{k}} (S_{\leq c})$ as well, because of the fact that
  retraction of $S_{\leq c}$ to $S_{\leq c-t} \cup_{B} A$ is chosen to be
  equivariant. The equality \eqref{eqn:inequality1} then follows immediately,
  since $G_c^-$ is empty in this case.
  
  We now consider the case when $\sum_{1 \leq i \leq k} \dfrac{\partial
  P}{\partial X_{i}} (\x)  < 0$. Let $T_{\x} W$ be the tangent space of $W$ at
  $\x$. The translation of $T_{\x} W$ to the origin is then the linear subspace
  $L \subset \R^{k}$ defined by $\sum_{i} X_{i} =0$. Let $L^{+} (\x) \subset
  L$ and $L^{-} (\x) \subset L$ denote the positive and negative eigenspaces
  of the Hessian of the function $e_{1}^{(k)}$ restricted to $W$ at $\x$.
  Let $\ind^{-} (\x) = \dim  L^{-} (\x)$, and let $\x  \in
  L_{\boldpi}$ where $\boldpi= (\pi^{(1)} , \ldots , \pi^{(
  \omega)}) \in \boldPi_{\mathbf{k}}$, where for each $i,1 \leq i
  \leq \omega$, $\pi^{(i)} = (\pi^{(i)}_{1} , \ldots , \pi^{(i
 )}_{\ell_{i}}) \in \Pi_{k_{i}}$. The subspaces $L^{+} (\x) ,L^{-} (\x
 )$ are stable under the the natural action of the subgroup $\prod_{1 \leq i
  \leq \omega ,1 \leq j \leq \ell_{i}} \mathfrak{S}_{\pi^{(i)}_{j}}$ of
  $\mathfrak{S}_{\mathbf{k}}$. For $1 \leq i \leq \omega ,1 \leq j \leq
  \ell_{i}$, let $L^{(i)}_{j}$ denote the subspace $L \cap L_{\pi^{(i
 )}_{j}}$ of $L$, and $M^{(i)}_{j}$ the orthogonal complement of $L^{(i
 )}_{j}$ in $L$. Let $L_{\fixed} =L \cap L_{\boldpi}$. It
  follows from 
  Parts \eqref{itemlabel:prop:orthogonal:2}, \eqref{itemlabel:prop:orthogonal:3}, and\eqref{itemlabel:prop:orthogonal:4}
  of Proposition \ref{prop:orthogonal} that:
  \begin{enumerate}[i]
    \item For each $i,j, \,1 \leq i \leq \omega ,1 \leq j \leq
    \ell_{i}$, $M^{(i)}_{j}$ is an irreducible representation of
    $\mathfrak{S}_{\pi_{i}}$, and the action of $\mathfrak{S}_{\pi^{(i'
   )}_{j'}}$ on $M^{(i)}_{j}$ is trivial if $(i,j) \neq (i' ,j')$.
    Hence, for each $i,   j   ,1 \leq i \leq \omega ,1 \leq j
    \leq \ell_{i}$, $L^{-} (p) \cap M^{(i)}_{j} =0  \tmop{or}  M^{(i
   )}_{j}$.
    
    \item The subgroup $\prod_{1 \leq i \leq \omega ,1 \leq j \leq \ell_{i}}
    \mathfrak{S}_{\pi^{(i)}_{j}}$ of $\mathfrak{S}_{\mathbf{k}}$ acts
    trivially on $L_{\fixed}$.
    
    \item There is an orthogonal decomposition $L=L_{\fixed} \oplus
    \left(\bigoplus_{1 \leq i \leq \omega ,1 \leq j \leq \ell_{i}} M^{(i
   )}_{j} \right)$.
  \end{enumerate}
  It follows that
  \begin{eqnarray*}
    L^{-} (p) & = & L'_{\fixed}   \oplus \left(\bigoplus_{(i,j) \in
    I} M^{(i)}_{j} \right) ,
  \end{eqnarray*}
  where $L'_{\fixed}$ is some subspace of $L_{\fixed}$ and $I
  \subset \{ (i,j) \mid 1 \leq i \leq \omega ,1 \leq j \leq \ell_{i} \}$.
  
  It follows from the proof of Proposition 7.19 in {\cite{BPRbook2}} that
  for all sufficiently small $t >0$ then $S_{\leq c}$ is retracts
  $\mathfrak{S}_{\mathbf{k}}$-equivariantly to a space $S_{\leq c-t} \cup_{B}
  A$ where the pair $(A,B) = \coprod_{\x} ( A_{\x} ,B_{\x})$, and the
  disjoint union is taken over the set critical points $\x$ with $F (\x) =c$,
  and each pair $(A_{\x} ,B_{\x})$ is homeomorphic to the pair $(\mathbf{D}^{\ind^{-} (\x)} , \partial \mathbf{D}^{\ind^{-} (\x
 )})$. It follows from the fact that the retraction mentioned above is
  equivariant that $\phi_{\mathbf{k}} (S_{\leq c})$ retracts to a space
  obtained from $\phi_{\mathbf{k}} (S_{\leq c-t})$ by gluing
  $\orbit_{\mathfrak{S}_{\mathbf{k}}} \left(\coprod_{\x} A_{\x} \right)$
  along $\orbit_{\mathfrak{S}_{\mathbf{k}}} \left(\coprod_{\x} B_{\x}
  \right)$. Now there are the following cases to consider:
  \begin{enumerate}[(a)]
    \item $\ind^{-} (\x) =0$. In this case
    \[
    \orbit_{\mathfrak{S}_{\mathbf{k}}} (\coprod_{\x} A_{\x},
    \coprod_{\x} B_{\x})
    \] 
    is homotopy equivalent to $(\ast , \emptyset)$.
    
    \item $L^{-} (\x)   \subset L_{\fixed}$ (i.e. $I= \emptyset$ in
    this case). In this case 
    \[
    \orbit_{\mathfrak{S}_{\mathbf{k}}}
    (\coprod_{\x} A_{\x} , \coprod_{\x} B_{\x})
    \] 
    is homeomorphic to
    $(\mathbf{D}^{\ind^{-} (\x)} , \partial \mathbf{D}^{\ind^{-}
    (\x)})$ by 
    Part \eqref{itemlabel:prop:orthogonal:5}
    of Proposition \ref{prop:orthogonal}.
    
    \item Otherwise, there is a non-trivial action on $L^{-} (\x)$ 
    of  the group 
    \[
    \prod_{(i,j) \in I} \mathfrak{S}_{\pi^{(i)}_{j}},
    \] 
    and it
    follows from 
    Part  \eqref{itemlabel:prop:orthogonal:5}
    of Proposition \ref{prop:orthogonal} that in this
    case 
    \[
    \orbit_{\mathfrak{S}_{\mathbf{k}}}(\coprod_{\x} A_{\x}, \coprod_{\x} B_{\x})
    \]
     is homotopy equivalent to $(\ast , \ast)$.
  \end{enumerate}
  The inequality \eqref{eqn:inequality1} 
  follow immediately from 
  inequality \eqref{eqn:MV2}.
  
  \noindent{Proof of Part \eqref{itemlabel:lem:equivariant_morseB:3}:}
  Follows from Part \eqref{itemlabel:prop:orthogonal:5}
  of Proposition \ref{prop:orthogonal}, and
 the fact that 
  \[
  \dim  L_{\fixed} = \length (\boldpi) -1,
  \] 
  by 
  Part \eqref{itemlabel:prop:orthogonal:1}
  of Proposition \ref{prop:orthogonal}.

This finishes the proof in case $\R =\mathbb{R}$. The statement over a
general real closed field $\R$ now follows by a standard application of the
Tarski-Seidenberg transfer principle (see for example the proof of Theorem
7.23 in {\cite{BPRbook2}}).\end{proof}

The proof of Lemma  \ref{lem:equivariant_morseB} is illustrated by the following simple
example.

\begin{figure}[h]
  \includegraphics[width=31pc]{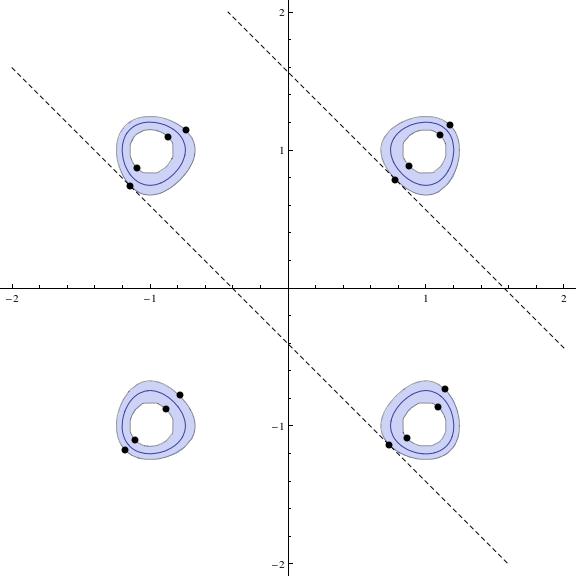}
\caption{\label{fig:figure1}  The real variety $\ZZ \left(P, \R^{2} \right)$, and the set defined by  $\Def (P, \zeta ,6) \leq 0$, in Example \ref{ex:example1}.}
\end{figure}

\begin{example}
\label{ex:example1}
  In this example, the number of blocks $\omega =1$, and $k=k_{1} =2$. 
  Consider the polynomial
  \begin{eqnarray*}
    P & = & (X_{1}^{2} -1)^{2} + (X_{2}^{2} -1)^{2} - \eps ,
  \end{eqnarray*}
  for some small $\eps >0$. The sets $\ZZ \left(P, \R^{2} \right)$, and $S=
  \left\{ x \in \R \langle \zeta \rangle^{2} \mid \bar{P} \leq 0 \right\}$,
  where $\bar{P} =  \Def (P, \zeta ,6)$ is shown in the Figure
  \ref{fig:figure1}.
  
  The polynomial $e_{1} (X_{1} ,X_{2}) =X_{1} +X_{2}$ has $16$ critical
  points, corresponding to $12$ critical values, $v_{1} < \cdots <v_{12}$, on
  $\ZZ \left(\bar{P} , \R \langle \zeta \rangle^{2} \right)$ of which $v_{5}$
  and $v_{9}$ are indicated in Figure \ref{fig:figure1} using dotted lines.
  The corresponding indices of the critical points, the number of critical
  points for each critical value, the sign of the polynomial $\dfrac{\partial
  \bar{P}}{\partial X_{1}} + \dfrac{\partial \bar{P}}{\partial X_{2}}$ at
  these critical points, and the partition $\pi \in \Pi_{2}$ such that the
  corresponding critical points belong to $L_{\pi}$ are shown in Table
  \ref{tab:table1}. The critical points corresponding to the shaded rows are
  the the critical points where $\left(\dfrac{\partial \bar{P}}{\partial
  X_{1}} + \dfrac{\partial \bar{P}}{\partial X_{2}} \right) <0$, and 
  these are the critical points whose orbits are represented in the sets $G_c^-$ in 
  Lemma \ref{lem:equivariant_morseB} above.
 
   {\tiny
  \begin{table}[h]
    \begin{tabular}{|l|l|l|l|l|l|l|}
      \hline
      Critical values & Index  & $ \SIGN \left(
      \dfrac{\partial \bar{P}}{\partial X_{1}} + \dfrac{\partial
      \bar{P}}{\partial X_{2}} \right)$ & $\pi$ & $L^{-} (p)$ &
      $L_{\fixed}$ & $L^{-} (p) \subset
      L_{\fixed}$\\
      \hline
      \rowcolor{LightCyan}$v_{1}$ & 0 & $-1$ & $(2)$ & $0$ & $0$ & yes\\
      \hline
      $v_{2}$ & 0 & $1$ & $(2)$ & 0 & $0$ & yes\\
      \hline\rowcolor{LightCyan}
      $v_{3}$ & 1 & $-1$ & $(2)$ & $L$ & $0$ & no\\
      \hline
      $v_{4}$ & 1 & $1$ & $(2)$ & $L$ & $0$ & no\\
      \hline\rowcolor{LightCyan}
      $v_{5}$ & $0$ & $-1$ & $(1,1)$ & $0$ & $L$ & yes\\
      \hline
      $v_{6}$ & 0 & 1 & $(1,1)$ & $0$ & $L$ & yes\\
      \hline\rowcolor{LightCyan}
      $v_{7}$ & $1$ & $-1$ & $(1,1)$ & $L$ & $L$ & yes\\
      \hline
      $v_{8}$ & $1$ & $1$ & $(1,1)$ & $L$ & $L$ & yes\\
      \hline\rowcolor{LightCyan}
      $v_{9}$ & $0$ & $-1$ & $(2)$ & $0$ & $0$ & yes\\
      \hline
      $v_{10}$ & 0 & $1$ & $(2)$ & $0$ & $0$ & yes\\
      \hline\rowcolor{LightCyan}
      $v_{11}$ & $1$ & $-1$ & $(2)$ & $L$ & $0$ & no\\
      \hline
      $v_{12}$ & $1$ & 1 & $(2)$ & $L$ & $0$ & no\\
      \hline
    \end{tabular}
    \caption{\label{tab:table1} Table of critical values in Example \ref{ex:example1}.}
  \end{table}
  }
  
  \end{example}

\begin{proof}[Proof of Proposition \ref{prop:betti_bound}] The
proposition follows directly from Lemmas \ref{lem:equivariant_morseA} and
\ref{lem:equivariant_morseB}, after noting that at most half the critical
values of $F$ satisfy 
\eqref{eqn:equivariant_morseB}
of Lemma
\ref{lem:equivariant_morseB}.
\end{proof}

\section{Proofs of the main theorems}\label{sec:proofs}

We are now in a position to prove the main theorems.

\subsection{Proof of Theorem \ref{thm:main}}
\begin{proof}[Proof of Theorem \ref{thm:main}] By Remark
\ref{rem:non-symm-to-symm}, we can assume without loss of generality that $P$
is symmetric in each block of variables $X^{(1)} , \ldots ,X^{(\omega)}$.
We first assume that $\ZZ \left(P, \R^{k} \right)$ is bounded. Let $d'$ be
the least even number such that $d' >d= \deg (P)$ and such that $d' -1$ is
prime. By Bertrand's postulate we have that $d'   \leq 2d$. Now, if $p$
divides $k$, replace $P$ by the polynomial
\[ P  \noplus \noplus +X_{k+1}^{2} ,  \]
and let $\omega' = \omega +1$, $k' =k+1$, and $\mathbf{k}' = (\mathbf{k},1
)$. Otherwise, let $\omega' = \omega +1$, $k' =k$, and $\mathbf{k}' = (
\mathbf{k},0)$. In either case, we have that $\gcd (p,k') =1$, and $k'
\leq k+1$.

Using Proposition \ref{prop:alg-to-semialg},
\begin{eqnarray*}
  b (\phi_{\mathbf{k}} (V) ,\F) & = & b (\phi_{\mathbf{k}'} (
  S) ,\F)
\end{eqnarray*}
where $S$ is the semi-algebraic set defined by $\Def (P, \zeta ,d') \leq
0$. It now follows from Propositions \ref{prop:non-degenerate},
\ref{prop:half-degree}, \ref{prop:betti_bound} , and Bezout's theorem that
\begin{eqnarray*}
  b (\phi_{\mathbf{k}'} (S) ,\F) & \leq & \frac{1}{2} 
  \sum_{\substack{\boldpi= (\pi^{(1)} , \ldots , \pi^{(\omega')}) \in
  \boldPi_{\mathbf{k}'}, \\ \length (\pi^{(i)}) \leq d', \\
  1 \leq i \leq \omega'}}  d' (d' -1)^{\length (\boldpi) -1} .
\end{eqnarray*}

After noting that using Bertrand's postulate $d' \leq 2d$, and using the
fact that $k' \leq k+1$, we obtain that in the bounded case,
\begin{eqnarray*}
  b (\phi_{\mathbf{k}} (V) ,\F) & \leq & \sum_{\substack{\boldsymbol{\ell}'
  = (\ell_{1} , \ldots , \ell_{\omega'}), \\
  1 \leq \ell_{i} \leq \min (
  k_{i},2d)}} p (\mathbf{k}' ,\boldsymbol{\ell}')  d (2d-1)^{| \boldsymbol{\ell}' |
  -1}
\end{eqnarray*}

Eqn. \eqref{eqn:thm:main1} follows from Eqn. \eqref{eqn:prop-main} in Proposition
\ref{prop:betti_bound}.

To take of the possibly unbounded case we introduce a new variable $Z$, and
let
\begin{eqnarray*}
  P_{1} & = & P+ \left(Z^{2} + \sum_{i=1}^{k} X_{i}^{2} + \sum_{j=1}^{m}
  Y_{j}^{2}  - \Omega^{2} \right)^{2}   ,\\
  P_{2} & = & P+ \left(\sum_{i=1}^{k} X_{i}^{2} + \sum_{j=1}^{m} Y_{j}^{2}  -
  \Omega^{2} \right)^{2} .
\end{eqnarray*}
Notice that, $V_{1} = \left(\ZZ \left(P_{1} , \R \left\langle
\tfrac{1}{\Omega} \right\rangle^{k+1} \right) \right)$ is semi-algebraically
homeomorphic to two homeomorphic copies of $V_{2} = \left(\ZZ \left(P_{2} ,
\R \left\langle \tfrac{1}{\Omega} \right\rangle^{k} \right) \right)$ glued
along $V$.

Using the fact that the map $\phi_{\mathbf{k}}$ is proper, it now follows from
inequality \eqref{eqn:MV2} that
\begin{eqnarray*}
  b (\phi_{\mathbf{k}} (V) ,\F) & \leq & \frac{1}{2} (b (
  \phi_{(\mathbf{k},1)} (V_{1}) ,\F) +b (\phi_{\mathbf{k}} (
  V_{2}) ,\F)) .
\end{eqnarray*}
Noticing that both $\ZZ \left(P_{1} , \R \left\langle \tfrac{1}{\Omega}
\right\rangle^{k+1} \right)$ and $\ZZ \left(P_{2} , \R \left\langle
\tfrac{1}{\Omega} \right\rangle^{k} \right)$ are bounded, we can use the
result from the bounded case and obtain in general that for $d \geq 4$,
\begin{eqnarray*}
  b (\phi_{\mathbf{k}} (V) ,\F) & \leq & \sum_{\substack{\boldsymbol{\ell}'
  = (\ell_{1} , \ldots , \ell_{\omega'}), \\1 \leq \ell_{i} \leq \min (
  k_{i},2d)} }(p (\mathbf{k}' ,\boldsymbol{\ell}'))  d (2d-1)^{| \boldsymbol{\ell}'
  |}\\
  & \leq & \sum_{\substack{\boldsymbol{\ell}= (\ell_{1} , \ldots , \ell_{\omega}), \\
  1 \leq \ell_{i} \leq \min (k_{i},2d)}}
   (p (\mathbf{k},\boldsymbol{\ell}))  d (
  2d-1)^{| \boldsymbol{\ell} | +1} ,
\end{eqnarray*}
where the last inequality follows from the fact that $\omega' = \omega +1$, and
$k_{\omega'} =1$. 
\end{proof}

\subsection{Proof of Theorem \ref{thm:main-sa-closed}}

\begin{definition}
  For any finite family $\mathcal{P} \subset \R [ X_{1} , \ldots ,X_{k} ]$ and
  $\ell \geq 0$, we say that $\mathcal{P}$ is in $\ell$-general position with
  respect to a semi-algebraic set $V \subset \R^{k}$ if for any subset
  $\mathcal{P}' \subset \mathcal{P}$, with $\card (
  \mathcal{P}') > \ell$, $\ZZ (\mathcal{P}' ,V) = \emptyset$. 
\end{definition}

Let $\mathbf{k}= (k_{1} , \ldots ,k_{\omega})$ with $k=
\sum_{i=1}^{\omega} k_{i}$, and 
\[
\mathcal{P} = \{ P_{1} , \ldots ,P_{s} \} \subset \R [\X^{(1)} , \ldots,\X^{(\omega)} ]^{\mathfrak{S}_\kk}
\]
be a fixed finite set of polynomials where $\X^{(i)}$ is a block of $k_{i}$
variables.
Let $\deg (P_{i}) \leq d$ for $1 \leq
i \leq s$. Let $\overline{\eps} = \left(\eps_{1} , \ldots , \eps_{s} \right)$
be a tuple of new variables, and let $\mathcal{P}_{\overline{\eps}} =
\bigcup_{1 \leq i \leq s} \left\{ P_{i}   \pm \eps_{i} \right\}$. We have the
following two lemmas.

\begin{lemma}
  \label{lem:gen-pos1-with-parameters}Let
  \begin{eqnarray*}
    D' (\mathbf{k},d) & = & \sum_{i=1}^{\omega} \min (k_{i} ,d) .
  \end{eqnarray*}
  The set of polynomials  $\mathcal{P}_{\overline{\eps}} \subset \R' [
  \X^{(1)} , \ldots ,\X^{(
  \omega)} ]$ is in $D'$-general position for any semi-algebraic subset $Z
  \subset \R^{k}$ stable under the action of $\mathfrak{S}_{\mathbf{k}}$,
  where $\R' = \R \langle \overline{\eps} \rangle$.
\end{lemma}

\begin{proof}
Using Lemma \ref{lem:multisymmetric-polynomial}
for each $i,1 \leq i \leq s$,  there exists
$\widetilde{P}_{i} \in \R [\ZB^{(1)} , \ldots ,\ZB^{(\omega)} ]$, where each $Z^{(
i)}$ is a block of $\ell_{i} = \min (k_{i} ,d)$ variables such that
\begin{eqnarray*}
  P_{i} & = & \widetilde{P}_{i} ((p_{1}^{(k_{1})} (\X^{(1)}) , \ldots
  ,p_{\ell_{1}}^{(k_{1})} (\X^{(1)})) , \ldots , (p_{1}^{(
  k_{\omega})} (\X^{(\omega)}) , \ldots ,p_{\ell_{\omega}}^{(
  k_{\omega})} (\X^{(\omega)}))) .
\end{eqnarray*}

Clearly,
$$\displaylines{
  P_{i} \pm \eps_{i}  =  \cr
   \widetilde{P}_{i} ((p_{1}^{(k_{1})} (\X^{(1)}
 ) , \ldots ,p_{\ell_{1}}^{(k_{1})} (\X^{(1)})) , \ldots , (
  p_{1}^{(k_{\omega})} (\X^{(\omega)}) , \ldots
  ,p_{\ell_{\omega}}^{(k_{\omega})} (\X^{(\omega)}))) \pm
  \eps_{i} .
}
$$
Since no sub-collection of the polynomials $\bigcup\limits_{1 \leq i \leq s} \left\{
\widetilde{P}_{i} \pm \eps_{i} \right\}$ of cardinality at least 
\[
1+\sum\limits_{i=1}^{\omega} \min
(k_{i} ,d)=D' +1
\] 
can have a common zero in $\R'^{D'}$, the lemma
follows.\end{proof}

Let $\Phi$ be a $\mathcal{P}$-closed formula, and let $S= \RR (\Phi ,V)$ be
bounded over $\R$. Let $\Phi_{\overline{\eps}}$ be the
$\mathcal{P}_{\overline{\eps}}$-closed formula obtained from $\Phi$ be
replacing for each $i,1 \leq i \leq s$,
\begin{enumerate}[i.]
  \item each occurrence of $P_{i} \leq 0$ by $P_{i} - \eps_{i} \leq 0
   $, and
  
  \item each occurrence of $P_{i} \geq 0$ by $P_{i} + \eps_{i} \geq 0
   $.
\end{enumerate}
Let $\R' = \R \left\langle \eps_{1} , \ldots , \eps_{s} \right\rangle$, and \
$S_{\overline{\eps}} = \RR \left(\Phi_{\overline{\eps}} , \R'^{k} \right)$.

\begin{lemma}
  \label{lem:gen-pos2-with-parameters} For any $r>0$, $r \in \R$, the
  semi-algebraic set  
  $\Ext(S \cap \overline{B_{k} (0,r)} , \R')$ is contained in 
  $S_{\overline{\eps}} \cap \overline{B_{k} (0,r)}$, and the inclusion 
  $\Ext(S \cap \overline{B_{k} (0,r)} , \R') \hookrightarrow S_{\overline{\eps}} \cap
  \overline{B_{k} (0,r)}$ is a semi-algebraic homotopy equivalence.
  
  Moreover, 
$\Ext(S \cap \overline{B_{k} (0,r)} , \R')
  /\mathfrak{S}_{\mathbf{k}}\subset (S_{\overline{\eps}} \cap
  \overline{B_{k} (0,r)}) /\mathfrak{S}_{\mathbf{k}}$, and the
  inclusion map 
  $\Ext(S \cap \overline{B_{k} (0,r)} , \R')/\mathfrak{S}_{\mathbf{k}} 
  \hookrightarrow \left(
  S_{\overline{\eps}} \cap \overline{B_{k} (0,r)} \right)
  /\mathfrak{S}_{\mathbf{k}}$ is a semi-algebraic homotopy equivalence.
\end{lemma}

\begin{proof} The proof is similar to the one of Lemma 16.17 in
{\cite{BPRbook2}}.
\end{proof}

\begin{remark}
  \label{rem:gen-pos3} In view of Lemmas \ref{lem:gen-pos1-with-parameters} and
  \ref{lem:gen-pos2-with-parameters} we can assume (at the cost of doubling the number of
  polynomials) after possibly replacing $\mathcal{P}$ by
  $\mathcal{P}_{\overline{\eps}}$, and $\R$ by 
  $\R'$,
   that the family $\mathcal{P}$ is in
  $D'(\mathbf{k},d)$-general position.
\end{remark}

Now, let $\delta_{1} , \cdots , \delta_{s}$ be new infinitesimals, and let
$\R'' = \R' \langle \delta_{1} , \ldots , \delta_{s} \rangle$.

\begin{notation}
  We define $\mathcal{P}_{>i} = \{P_{i+1} , \ldots ,P_{s} \}$ and
  \begin{eqnarray*}
    \Sigma_{i} & = & \{P_{i} =0,P_{i} = \delta_{i} ,P_{i} = - \delta_{i}
    ,P_{i} \geq 2 \delta_{i} ,P_{i} \leq -2 \delta_{i} \} ,\\
    \Sigma_{\le i} & = & \{\Psi \mid \Psi = \bigwedge_{j=1, \ldots ,i}
    \Psi_{i} , \Psi_{i} \in \Sigma_{i} \} .
  \end{eqnarray*}
  Note that for each $\Psi \in \Sigma_{i}$, $\RR (\Psi , \R'\la\delta_1,\ldots,\delta_i{k})$ is symmetric with respect
  to the action of $\mathfrak{S}_{\kk}$ 
  and for if
  $\Psi \neq \Psi'$, $\Psi , \Psi' \in \Sigma_{\leq i}$,
  \begin{eqnarray}
  \label{eqn:disjoint}
    \RR(\Psi , \R' \langle \delta_{1} , \ldots , \delta_{i} \ra^{k})
    \cap \RR(\Psi' , \R'\la \delta_{1} , \ldots , \delta_{i} \ra^{k}) & = & \emptyset .  
  \end{eqnarray}

  If $\Phi$ is a $\mathcal{P}$-closed formula, we denote
  \begin{eqnarray*}
    \RR_{i} (\Phi) & = & \RR(\Phi , \R' \la \delta_{1} , \ldots ,
    \delta_{i} \ra^{k}) ,
  \end{eqnarray*}
  and
  \begin{eqnarray*}
    \RR_{i} (\Phi \wedge \Psi) & = & \RR(\Psi , \R' \la \delta_{1} ,
    \ldots , \delta_{i} \ra^{k}) \cap \RR_{i} (\Phi) .
  \end{eqnarray*}
  Finally, we denote for each $\mathcal{P}$-closed formula $\Phi$
  \begin{eqnarray*}
    b (\Phi /\mathfrak{S}_{\mathbf{k}} ,\F) & = & b(\RR
   (\Phi , \R''^{k}) /\mathfrak{S}_{\mathbf{k}} ,\F
   ) .
  \end{eqnarray*}
\end{notation}

The proof of the following proposition is very similar to Proposition 7.39 in
{\cite{BPRbook2}} where it is proved in the non-symmetric case.

\begin{proposition}
  \label{7:prop:closed-with-parameters}For every $\mathcal{P}$-closed formula
  $\Phi$,
  such that $\RR(\Phi,\R^k)$ is bounded,
  \begin{eqnarray*} b (\Phi /\mathfrak{S}_{\mathbf{k}} ,\F) &\leq&
     \sum_{\substack{
       \Psi \in \Sigma_{\le s}\\
       \RR_{s} (\Psi , \R''^{k}) \subset \RR_{s} (\Phi , \R''^{k})}}
      b (\Psi /\mathfrak{S}_{\mathbf{k}} ,\F).
      \end{eqnarray*}
\end{proposition}

\begin{proof}First observe that the orbit space of a disjoint
union of symmetric sets is a disjoint union of the corresponding orbit spaces.
The symmetric semi-algebraic sets  $\RR(\Psi ,\R''^k) ,
\Psi \in \Sigma_{\leq s}$ are disjoint by \eqref{eqn:disjoint}. The proof is
now the same as the proof of Proposition 7.39 in {\cite{BPRbook2}}.
\end{proof}

Let
\begin{eqnarray*}
  D' =D' (\mathbf{k},d) & = & \sum^{\omega}_{i=1} \min (k_{i} ,d) ,\\
  D'' =D'' (\mathbf{k},d) & = & \sum^{\omega}_{i=1} \min (k_{i} ,4d) .
\end{eqnarray*}
\begin{proposition}
  \label{7:prop:betti closed}
  
  For $0 \leq i<D' +D''$,
  \[ \sum_{\Psi \in \Sigma_{\le s}} b_{i} (\Psi /\mathfrak{S}_{\kk}
     ,\F) \leq \sum_{j=0}^{D' +D'' -i} \binom{s}{j} 6^{j} F (
     \mathbf{k},2 d) . \]
  For $i \geq D' +D''$,
  \[ \sum_{\Psi \in \Sigma_{\le s}} b_{i} (\Psi /\mathfrak{S}_{\kk}
     ,\F) =0. \]
  
\end{proposition}

We first prove the following lemmas.
Let $Q_{i} =P_{i}^{2} (P_{i}^{2} - \delta_{i}^{2})^{2} (P_{i}^{2} -4
\delta_{i}^{2})$.

For $j \ge 1$ let,
\begin{eqnarray*}
  V_{j} & = & \RR(\bigvee_{1 \leq i \leq j} Q_{i} =0, \R' \la
  \delta_{1} , \ldots , \delta_{j} \ra^{k}) ,\\
  W_{j} & = & \RR(\bigvee_{1 \leq i \leq j} Q_{i} \geq 0, \R' \la
  \delta_{1} , \ldots , \delta_{j} \ra^{k}) .
\end{eqnarray*}
\begin{lemma}
  \label{lem:gen-position} Let $I \subset [ 1,s ]$, $\sigma  =  (\sigma_{1} ,
  \ldots , \sigma_{s}) \in \{ 0, \pm 1, \pm 2 \}^{s}$ and let
  \[
  \mathcal{P}_{I, \sigma} =  \bigcup_{i \in I} \{ P_{i} + \sigma_{i}
  \delta_{i} \}.
  \] 
  Then, $\ZZ(\mathcal{P}_{I, \sigma}, \R''^k) = \emptyset$,
  whenever $\card (I) >D'$.
\end{lemma}

\begin{proof}
This follows from the fact that $\mathcal{P}$ is in
$D'$-general position by Remark
\ref{rem:gen-pos3}.
\end{proof}

\begin{lemma}
  \label{7:lem:union2} For each $i,0 \leq i< D' +D''$,
  \[ 
  b_{i} (V_{j} /\mathfrak{S}_{\kk} ,\F) \leq (6^{j} -1) F (
     \mathbf{k},2d) . 
     \]
  For $i \geq D' +D''$,
  \begin{eqnarray*}
    b_{i} (V_{j} /\mathfrak{S}_{\kk} ,\F) & = & 0.
  \end{eqnarray*}
\end{lemma}

\begin{proof}
Clearly, $V_j$
is the disjoint union of the real varieties
$$
\displaylines{
\ZZ(P_{i}, \R' \la \delta_{1} , \ldots , \delta_{j} \ra^{k}), \cr
\ZZ (P_{i} \pm\delta_{i} , \R' \la \delta_{1} , \ldots , \delta_{j} \ra^{k}),\cr
\ZZ (P_{i} \pm 2 \delta_{i} , \R' \la \delta_{1} , \ldots , \delta_{j}\ra^{k}),
 }
 $$
for $1 \leq i \leq j$,
and hence the quotient
$V_j/\mathfrak{S}_\kk$
is the disjoint union of the quotients
\begin{eqnarray}
\nonumber
&\ZZ(P_{i}, \R' \la \delta_{1} , \ldots , \delta_{j} \ra^{k})/\mathfrak{S}_\kk,& \\
\label{eqn:list}
&\ZZ (P_{i} \pm\delta_{i} , \R' \la \delta_{1} , \ldots , \delta_{j} \ra^{k})/\mathfrak{S}_\kk,&\\
\nonumber
&\ZZ (P_{i} \pm 2 \delta_{i} , \R' \la \delta_{1} , \ldots , \delta_{j}\ra^{k})/\mathfrak{S}_\kk.&
 \end{eqnarray}

It follows from Part \eqref{itemlabel:7:prop:prop1:1}  of
Proposition \ref{7:prop:prop1} that $b_{i} (V_{j}
/\mathfrak{S}_{\kk} ,\F)$ is bounded by the sum for $1 \leq \ell \leq
i+1$, of $(i- \ell +1)$-th Betti numbers of all possible $\ell$-ary
intersections  amongst quotients of the  
varieties listed above.
It is clear that
the total number of such non-empty $\ell$-ary intersections is at most
$\binom{j}{\ell} 5^{\ell}$. It now follows from Theorem \ref{thm:main} applied
to the non-negative symmetric polynomials $P_{i}^{2} , (P_{i} \pm
\delta_{i})^{2} , (P_{i} \pm 2 \delta_{i})^{2}$, and noting that the
degrees of these polynomials are bounded by $2d$, that
\begin{eqnarray*}
  b_{i} (V_{j} /\mathfrak{S}_{\kk} ,\F) & \leq & \sum_{p=1}^{\min (
  j,D')} \binom{j}{p} 5^{p} F (\mathbf{k},2d) .
\end{eqnarray*}
To prove the vanishing of the higher Betti numbers, first observe that $(i-
\ell +1)$-th Betti numbers of all possible $\ell$-ary intersections amongst
the sets listed in \eqref{eqn:list} vanish for $i- \ell +1>D''$ using Theorem
\ref{thm:main}.

Also, notice that by Lemma \ref{lem:gen-position} the $\ell$-ary intersections
amongst the sets in \eqref{eqn:list} are empty for $\ell >D'$. Together, these
observations imply that
\begin{eqnarray*}
  b_{i} (V_{j} /\mathfrak{S}_{\kk} ,\F) & = & 0.
\end{eqnarray*}
for all $i \geq D' +D''$. To see this observe that if $i \geq D' +D''$, and
$\ell \leq D'$, then $i- \ell +1 \geq  D'' +1$.
\end{proof}

\begin{lemma}
  \label{7:lem:union1}
  For $0 \leq i<D' +D'' $,
  \[
   b_{i} (W_{j} /\mathfrak{S}_{\kk} ,\F) \leq \sum_{p=1}^{\min (
     j,D')} \binom{j}{p} 5^{p}  (F (\mathbf{k},2 d)) +b_{i}(\R' \la
     \delta_{1} , \ldots , \delta_{j} \ra^{k} /\mathfrak{S}_{k} ,\F
    ) . \]
  For $i \geq D' +D''$, $b_{i} (W_{j} /\mathfrak{S}_{\kk} ,\F) =0$.
 \end{lemma}

\begin{proof} Let
\[ 
W'_j= \RR (\bigwedge_{1 \leq i \leq j} Q_{i} \leq 0 \vee \bigvee_{1 \leq
   i \leq j} Q_{i} =0, \R' \la \delta_{1} , \ldots , \delta_{j} \rangle)^k
 ) . \]

Now, from the fact that 
\[ 
W_{j} \cup W'_j= \R' \la \delta_{1} , \ldots ,
   \delta_{j} \ra^{k} ,W_{j} \cap W' =V_{j},
\] 
it follows immediately that
\[ (W_{j} \cup W'_j) /\mathfrak{S}_{\mathbf{\kk}} = (W_{j}
   /\mathfrak{S}_{\mathbf{k}})\cup (W'_j/\mathfrak{S}_{\mathbf{k}}) = \R' \la
   \delta_{1} , \ldots , \delta_{j} \ra^{k} /\mathfrak{S}_{\mathbf{k}}, 
   \]
and
\[( W_{j} /\mathfrak{S}_{\mathbf{k}}) \cap (W'_j/\mathfrak{S}_{\mathbf{k}}) = 
(W_{j} \cap W'_j) /\mathfrak{S}_{\mathbf{k}} =V_{j}/\mathfrak{S}_{\mathbf{k}}. 
\]
Using 
inequality \eqref{eqn:MV2}
we get that
\begin{eqnarray*}
  b_{i} (W_{j} /\mathfrak{S}_{\mathbf{k}} ,\F) & \leq & b_{i} ((
  W_{j} \cap W'_j) /\mathfrak{S}_{\mathbf{k}} ,\F) +b_{i} ((W_{j}
  \cup W'_j) /\mathfrak{S}_{\mathbf{k}} ,\F)\\
  & = & b_{i} (V_{j} /\mathfrak{S}_{\mathbf{k}} ,\F) +b_{i}(
  \R' \la \delta_{1} , \ldots , \delta_{j} \ra^{k} /\mathfrak{S}_{\mathbf{k}}
  ,\F)
\end{eqnarray*}
We conclude using Lemma \ref{7:lem:union2}.
\end{proof}

\begin{proof}[Proof of Proposition \ref{7:prop:betti closed}]
Using Part \eqref{itemlabel:7:prop:prop1:2} of Proposition \ref{7:prop:prop1} we get that
\begin{eqnarray*}
  \sum_{\Psi \in \Sigma_{\le s}} b_{i} (\Psi /\mathfrak{S}_{\mathbf{k}}
  ,\F) & \leq & \sum_{j=1}^{k-i}
  \sum_{\substack{
    J \subset \{ 1, \ldots ,s \}\\
    \card (J) =j
 }} b_{i+j-1} (S^{J} /\mathfrak{S}_{\mathbf{k}},\F) + \binom{s}{k-i}
  b_{k} (S^{\emptyset} /\mathfrak{S}_{\mathbf{k}},\F) .
\end{eqnarray*}
It follows from Lemma \ref{7:lem:union1} that,
\begin{eqnarray*}
  b_{i+j-1} (S^{J} /\mathfrak{S}_{\mathbf{k}},\F) & = & 0,
\end{eqnarray*}
when $i+j-1  \geq D' +D''$, and otherwise,
\begin{eqnarray*}
  b_{i+j-1} (S^{J}/\mathfrak{S}_\kk,\F) & \leq & \sum_{\ell =1}^{\min (j,D')} \binom{j}{\ell}
  5^{\ell} F (\mathbf{k},2 d) +b_{k}(\R^{k}
  /\mathfrak{S}_{\mathbf{k}} ,\F) .
\end{eqnarray*}
Hence,
\begin{eqnarray*}
  \sum_{\Psi \in \Sigma_{\le s}} b_{i} (\Psi /\mathfrak{S}_{\kk} ,\F)
  & \leq & \sum_{j=1}^{D' +D'' -i}
  \sum_{\substack{
    J \subset \{ 1, \ldots ,s \}\\
    \card(J) =j
 }}  b_{i+j-1} (S^{J} /\mathfrak{S}_{\mathbf{k}}) +\\
 && \binom{s}{k-i}
  b_{k} (S^{\emptyset} /\mathfrak{S}_{\mathbf{k}})\\
  & \leq & \sum_{j=1}^{D' +D''  -i} \binom{s}{j} \left(\sum_{p=1}^{\min (
  j,D')} \binom{j}{p} 5^{p}  F (\mathbf{k},2d) \right)\\
  & \leq & \sum_{j=1}^{D' +D''  -i} \binom{s}{j} 6^{j} F (\mathbf{k},2d) .
\end{eqnarray*}
Finally, it is clear that
\begin{eqnarray*}
  \sum_{\Psi \in \Sigma_{\le s}} b_{i} (\Psi /\mathfrak{S}_{\mathbf{k}}
  ,\F) & = & 0,
\end{eqnarray*}
for $i  \geq D+D'$.

\end{proof}

\begin{proof}[Proof of Theorem \ref{thm:main-sa-closed}] 
We add an extra polynomial, $\delta(X_1^2+\cdots+X_k^2) -1$ to the set $\mathcal{P}$, replace the field $\R$, by
$\R\la\delta\ra$, and replace the given formula $\mathcal{P}$-closed formula $\Phi$ by the formula
$\Phi \wedge ( \delta(X_1^2+\cdots+X_k^2) -1\leq 0)$. Notice that the new set $\RR(\Phi)$
is bounded in $\R\la\delta\ra^k$ and has isomorphic homology groups as $S$.

We first
consider the case in which for each $i,1 \leq i \leq \omega$, $4d \leq k_{i}$.
In this case,
\begin{eqnarray*}
  D (\mathbf{k},d) & = & D' (\mathbf{k},d) +D'' (\mathbf{k},d) ,
\end{eqnarray*}
and Theorem \ref{thm:main-sa-closed} follows from Propositions
\ref{7:prop:closed-with-parameters} and \ref{7:prop:betti closed}, recalling
that the number of polynomials was doubled in ordered to put the family
$\mathcal{P}$ in $D'$-general position. In the general case, suppose without
loss of generality that $k_{i} \leq 4d$, for $1 \leq i \leq \omega' \leq
\omega$, and $k_{i} >4d  $ for $i> \omega'$. Let, $\mathbf{k}' = (
k_{1} , \ldots ,k_{\omega'})$, $k' = \sum_{i=1}^{\omega'} k_{i}$, and $\bar{\pi}
: \R^{k} /\mathfrak{S}_{\mathbf{k}} \rightarrow \R^{k'}
/\mathfrak{S}_{\mathbf{k}'}$ the map induced by the projection map, $\pi:\R^k\rightarrow \R^{k'}$, to the first $k'$ coordinates.

Then, for each $\bar{\mathbf{y}} \in \bar{\pi} (S /\mathfrak{S}_{\mathbf{k}})$,
we have by applying the special case of Theorem \ref{thm:main-sa-closed}
already proved above that,
\begin{eqnarray*}
  b_{i} ((S \cap \pi^{-1} (\bar{\mathbf{y}})) /\mathfrak{S}_{\mathbf{k}})
  ,\F) & = & 0,
\end{eqnarray*}
for $i \geq \sum_{i= \omega' +1}^{\omega} \left(\min (k_{i} ,4d) + \min (k_{i} ,d
)
\right) 
=5 (\omega - \omega') d$. 

In other words, the the fibers of the map $\bar{\pi}
: \R^{k} /\mathfrak{S}_{\mathbf{k}} \rightarrow \R^{k'}
/\mathfrak{S}_{\mathbf{k}'}$ restricted to $S/\mathfrak{S}_k$ have vanishing homology above 
(and including) dimension $5(\omega-\omega')d$, and clearly 
the image of the map has  dimension $\leq k'$.

It now follows from Leray spectral sequence of the map 
$\bar{\pi}
: \R^{k} /\mathfrak{S}_{\mathbf{k}} \rightarrow \R^{k'}
/\mathfrak{S}_{\mathbf{k}'}$ restricted to $S/\mathfrak{S}_k$ (see for example \cite[Th\'{e}or\`{e}me 5.2.4] {Godement}),
 that
\begin{eqnarray*}
  b_{i} (S /\mathfrak{S}_{\mathbf{k}}
  ,\F) & = & 0,
\end{eqnarray*}
for $i  \geq k' +5 (\omega - \omega') d=D (\mathbf{k},d)$.

A similar argument proves that in Proposition \ref{7:prop:betti closed} we
can replace $D' +D''$ by $D$ as well.

This proves the theorem in general.\end{proof}

\subsection{Proof of Theorem \ref{thm:main-sa}}

In {\cite{GV07}}, Gabrielov and Vorobjov introduced a construction for
replacing an arbitrary $\mathcal{P}$-semi-algebraic set $S$ by a certain
$\mathcal{P}'_{p}$-closed semi-algebraic set $S_{p}'$ (for any given $p \geq 0$), 
such that $S$ and $S_{p}'$ are $p$-equivalent. The family
$\mathcal{P}'_{p}$ in their construction is given by
\begin{eqnarray*}
  \mathcal{P}_{p}' & = & \bigcup_{P \in \mathcal{P}} \bigcup_{0 \leq i \leq p}
  \left\{ P \pm \eps_{i} ,P \pm \delta_{i} \right\} ,
\end{eqnarray*}
where the $\eps_{i} , \delta_{i}$ are infinitesimals.

Note that  $\mathcal{P} \subset \R [ \X^{(1)} ,\ldots ,\X^{(\omega)} ]^{\mathfrak{S}_\kk}$ implies that
$\mathcal{P}'_{p} \subset\R [ \X^{(1)} ,\ldots ,\X^{(\omega)} ]^{\mathfrak{S}_\kk}$ as well,
and if the degrees of the polynomials in $\mathcal{P}$ are
bounded by $d$, the same bound applies to polynomials in $\mathcal{P}'_{p}$ as well.
Furthermore, $\card (\mathcal{P}'_{p})  = 4 (p+1
) \card (\mathcal{P})$. It is an immediate
consequence of the above result that $S/\mathfrak{S}_{\mathbf{k}}$ is
$p$-equivalent to $S' /\mathfrak{S}_{\mathbf{k}}$ as well.

\begin{proof}[Proof of Theorem \ref{thm:main-sa}] Using the above
construction, replace $S$ by $S_{p}'$, with $p = k$. Then, apply Theorem
\ref{thm:main-sa-closed}.\end{proof}

\subsection{Proof of Theorem \ref{thm:descent2-quantitative}
}
We now prove Theorem
\ref{thm:descent2-quantitative} 
closely
following the proof of Theorem \ref{thm:descent-quantitative} in {\cite{GVZ04}}. We first
need a few preliminary definitions and notation.

For the rest of this section we fix $X$ to be a compact semi-algebraic subset of $\R$.
\begin{notation}[Standard simplex]
  We will denote by $\Delta_{p}$, the standard $p$-dimensional simplex,
  namely
  \begin{eqnarray*}
    \Delta_{p} & = & \{ (s_{0} , \ldots ,s_{p}) | s_{0} , \ldots ,s_{p} \geq
    0,s_{0} + \cdots +s_{p} =1 \} .
  \end{eqnarray*}
\end{notation}

\begin{notation}[Symmetric product]
\label{not:symmetric-product}
  We denote for each $p \geq 0$, $ \Sym^{(p)} (X
 )$ the $(p+1)$-fold symmetric product of $X$ i.e.
  \begin{eqnarray*}
     \Sym^{(p)} (X) & = & \underbrace{X \times
    \cdots \times X}_{p+1} /\mathfrak{S}_{p+1}. 
\end{eqnarray*}
  Let $\W^{(p)}  = \{\x = (x_0,\ldots,x_p) \in \R^{p+1} \mid x_0 \leq x_1 \leq \cdots \leq x_{p} \}$
  Then,   $\Sym^{(p)}(X)$ is homeomorphic to $X^{p+1} \cap \W^{(p)}$, 
  and we will identify $\Sym^{(p)}(X)$ with the set $X^{p+1} \cap \W^{(p)}$.
\end{notation}

\begin{definition}[Symmetric join]
We next define $J^{(p)}_\symm(X)$ as follows. 
\begin{eqnarray*}
    J^{(p)}_\symm (X) & = & 
    \Sym^{(p)}(X)
    \times \Delta_{p} / \sim,
  \end{eqnarray*}
where the equivalence relation $\sim$ is given by 
(after identifying $\Sym^{(p)}(X)$ with $X^{p+1} \cap \W^{(p)}$  
cf. Notation \ref{not:symmetric-product})
\[((x_{0} ,
  \ldots ,x_{p}) , (s_{0} , \ldots ,s_{p}))\sim ((x_{0}'
  , \ldots ,x_{p}'), (s_{0}' , \ldots ,s_{p}'))
  \] 
  if and only if $(
  s_{0} , \ldots ,s_{p}) = (s_{0}' , \ldots ,s_{p}') $, and $x_{i}
  =x_{i}'$ for all $i$ such that $s_{i} =s_{i}' \neq 0$.
  \end{definition}

For each $p >0, 0 \leq i \leq p$,
there is an injection
\[
\phi^{(p,i)} :J^{(p-1)}_{\symm} (X) \rightarrow J^{(p)}_{\symm} (X)
\]
defined by
$$\displaylines{
 \phi^{(p,i)} (((x_{0} , \ldots ,x_{p-1}) , (
  s_{0} , \ldots ,s_{p-1}))) = \cr
 ((x_{0} , \ldots, x_{i},x_{i},x_{i+1}, \ldots,
  ,x_{p-1}), (s_{0} , \ldots, s_{i-1},0,s_{i+1},\ldots,s_{p})).
}
$$

Let $J_\symm(X)$ be the disjoint union of the $J_\symm^{(p)}(X)$ with for each $p\geq 0$, the images of $\phi^{(p,i)}, 0 \leq i \leq p$, identified. Let 
\[
\phi^{(p)}: J_\symm^{(p-1)}(X) \rightarrow J_\symm(X)
\] 
be the maps induced by the $\phi^{(p,i)}$.

\begin{lemma}
  \label{lem:contractible} The image $\phi^{(p,i)} (J^{(p-1
 )}_{\symm} (X))$ is contractible inside 
 $J^{(p)}_{\symm} (X))$.
\end{lemma}

\begin{proof}
Without loss of generality, let $i=0$, and let $y = \min X$.
For each $t  \in [ 0,1 ]$, we define a map
$g_t: \phi^{(p,0)}(J^{(p-1)}_\symm(X)) \rightarrow J^{(p)}_\symm(X)$ as follows.
Let 
$$\displaylines{
((x_{0} , x_{0}, \ldots ,x_{p-1}), (0,s_{1} , \ldots ,s_{p})) = \cr
\phi^{(p,0)}((x_{0}, \ldots ,x_{p-1}), (s_{1} , \ldots ,s_{p})) \in
\phi^{(p,0)}(J^{(p-1)}_\symm(X)).
}
$$

  We define  
$$\displaylines{
  g_{t} ((x_{0} , x_{0}, \ldots ,x_{p-1}), (0,s_{1} , \ldots ,s_{p})))  = \cr
   ((y, x_{0},\ldots,x_{p-1}), (t,  (1-t)s_1\ldots,(1-t)s_{p})).
 }
 $$
Observe that, $g_{t}$ is a continuous family of maps, satisfying
\begin{eqnarray*}
g_{0} &=&
\tmop{Id}_{\phi^{(p,0)} (J^{(p-1)}_{ \symm} (X))}, \\
g_{1} (\phi^{(p,0)} (J^{(p-1)}_{ \symm} (X))) &=& ((y,\ldots,y),(1,0,\ldots,0)),
\end{eqnarray*}
proving the lemma.
\end{proof}

It follows immediately from Lemma \ref{lem:contractible} that

\begin{lemma}
  \label{lem:contractible-infinite-join}
  $J_{\symm} (X)$ is contractible.
\end{lemma}

Now suppose that $S \subset \R^{m+1}$ is a compact semi-algebraic set, and $\pi:S \rightarrow T = \pi(S)$ is the projection on the first $m$ coordinates restricted to $S$.

\begin{notation}
\label{not:fibered-symmetric-join}
We denote for each $p \geq 0$, $ \Sym^{(p)}_\pi(S)$ the $(p+1)$-fold symmetric product of $S$ fibered over $\pi$ i.e.
  \begin{eqnarray*}
     \Sym^{(p)}_\pi (S) & = & \underbrace{S \times_\pi
    \cdots \times_\pi S}_{p+1} /\mathfrak{S}_{p+1}. 
   \end{eqnarray*} 
   
   As before  we identify  $ \Sym^{(p)}_\pi(S)$ with the set 
   \[\W_\pi^{(p)} = \{(\y,x_0,\ldots,x_p) \mid (\y,x_i) \in S, 0\leq i \leq p, (x_0,\ldots,x_p) \in \W^{(p)}\}.
   \]
\end{notation}

\begin{definition}[Fibered symmetric join]
For each $p \geq 0$, we denote by $J^{(p)}_{\pi,\symm} (S)$, the $(p+1)$-fold fibered 
  symmetric join as the set defined by 
  \begin{eqnarray*}
    J^{(p)}_{\pi,\symm} (S) & = & 
    \Sym_\pi^{(p)}(S)
    \times \Delta_{p} / \sim,
  \end{eqnarray*}
where the equivalence relation $\sim$ is given by 
\[
((\y,x_{0} ,
  \ldots ,x_{p}) , (s_{0} , \ldots ,s_{p}))\sim ((\y,x_{0}',
  , \ldots ,x_{p}'), (s_{0}' , \ldots ,s_{p}'))
  \]
 if and only if $\y = \y', (s_{0} , \ldots ,s_{p}) = (s_{0}' , \ldots ,s_{p}')$, and $x_{i}
  =x_{i}'$ for all $i$ such that $s_{i} =s_{i}' \neq 0$.

For each $p >0, 0 \leq i \leq p$,
there is an injection
\[
\phi^{(p,i)}_\pi :J^{(p-1)}_{\pi,\symm} (S) \rightarrow J^{(p)}_{\pi,\symm} (S)
\]
defined by
$$\displaylines{
 \phi^{(p,i)}_\pi (((\y,x_{0}, \ldots ,x_{p-1}) , (
  s_{0} , \ldots ,s_{p-1}))) = \cr
 ((\y,x_{0} , \ldots,x_{i},x_{i},x_{i+1}, \ldots,
  ,x_{p-1}), 
   (s_{0} , \ldots, s_{i-1},0,s_{i+1},\ldots,s_{p})).
}
$$

Let $J_{\pi,\symm}(S)$ be the disjoint union of the $J_{\pi,\symm}^{(p)}(S)$ with for each $p> 0$, the images of $\phi^{(p,i)}_\pi, 0 \leq i \leq p $ identified. Let 
\[
\phi^{(p)}_\pi: J_{\pi,\symm}^{(p-1)}(S) \rightarrow J_{\pi,\symm}(S)
\] 
be the inclusion maps induced by the $\phi^{(p,i)}_\pi$.
\end{definition}

\begin{proposition}
  \label{prop:he}
  The induced surjection $J (\pi)
  :J_{\pi,\symm} (S) \twoheadrightarrow T$ is a
  homotopy equivalence.
\end{proposition}

\begin{proof} For each $\y \in T$, $J (\pi)^{-1} (\y)  =
J_{\symm} (\pi^{-1} (\y))$. By Lemma
\ref{lem:contractible-infinite-join}, 
\[
J_{ \symm} (\pi^{-1} (\y))
\]
is
contractible. The proposition now follows from the Vietoris-Begle
theorem.\end{proof}

\begin{lemma}
  \label{lem:quotient}The pair $(J^{(p)}_{\pi
  ,\symm} (S) , \phi^{(p)} (J^{(p-1
 )}_{\pi,\symm} (S)))$ is homotopy equivalent to
  the pair $(\mathbf{S}^{p} \times  \Sym^{(p
 )}_{\pi} (S) , \{ \ast \} \times  \Sym^{(p)}_{\pi}
  (S))$,
  where $\mathbf{S}^p$ denotes the $p$-dimensional sphere.
\end{lemma}

\begin{proof}Clear from the definition of $J^{(p)}_{\pi
,  \symm} (S)$, and the inclusion map $\phi^{(p)}$.
\end{proof}

\begin{theorem}
  \label{thm:symmetric-spectral-sequence}For any field of coefficients
  $\F$, there exists a spectral sequence converging to $\HH_{\ast} (
  T,\F)$ whose $E^{1}$-term is given by
  \begin{eqnarray}
    E^{1}_{p,q} & \simeq & \HH_{q} ( \Sym^{(p)}_{\pi}
    (S) ,\F) .  \label{eqn:E1}
  \end{eqnarray}
\end{theorem}

\begin{proof}The spectral sequence is the spectral sequence of
the filtration (see, for example, \cite[\S 4]{Godement})
\[ \tmop{Im} (\phi^{(0)}) \subset \tmop{Im} (\phi^{(1
  )}) \subset \cdots \subset J_{\pi,  \symm} (S) \sim T \]
where the last homotopy equivalence is a consequence of Proposition
\ref{prop:he}. The isomorphism in \eqref{eqn:E1}  is a consequence of Lemma
\ref{lem:quotient} after noticing that
\begin{eqnarray*}
  \HH_{q} ( \Sym^{(p)}_{\pi} (S) ,\F) & \simeq & \HH_{q-p} (
 \mathbf{S}^{p} \times  \Sym^{(p)}_\pi (S) , \{ \ast \} \times
   \Sym^{(p)}_{\pi} (S) ,\F)   ,q \geq p,\\
 \HH_{q} ( \Sym^{(p)}_{\pi} (S) ,\F) & \simeq & 0,q<p.
\end{eqnarray*}

\end{proof}

\begin{remark}
\label{rem:symmetric-spectral-sequence}
Similar spectral sequences for finite maps have been considered by several other authors 
(see for example \cite{Houston2,Goryunov-Mond}). 
The $E^1$-term of these spectral sequences involve the \emph{alternating
cohomology} of the fibered product, rather than the ordinary homology of 
the symmetric product as in Theorem \ref{thm:symmetric-spectral-sequence}. This distinction 
is important for us, as
we can apply our bounds on the equivariant Betti numbers of symmetric semi-algebraic sets
to bound the dimensions of the latter groups, but not those of the former.
\end{remark}

\begin{corollary}
  \label{cor:descent2} With the above notation and for any field of
  coefficients $\F$
  \begin{eqnarray*}
    b (\pi (V) ,\F) & \leq & \sum_{0 \leq p<m} b (
     \Sym^{(p)}_{\pi} (V) ,\F) .
  \end{eqnarray*}
\end{corollary}

\begin{proof}[Proof of Theorem \ref{thm:descent2-quantitative}] 
First observe that  
\[
\Sym^{(p)}_\pi(V) = \ZZ(Q^{(p)},\R^{(p+1)+m})/\mathfrak{S}_{\kk_{m,p}},
\]
where
\[
Q^{(p)} = \sum_{0 \leq i \leq p} P(\Y,X_{i}),
\]
and 
\[
\kk_{m,p} = {(\underbrace{1,\ldots,1}_{m},p+1)}.
\]
Note that $Q^{(p)}$ is symmetric in $\X^{(p)} = (X_0,\ldots,X_p)$, and thus 
\[
Q^{(p)} \in \R[\Y,\X^{(p)}]^{\mathfrak{S}_{\kk_{m,p}}}.
\]
Moreover, $Q^{(p)}$ is non-negative (since $P$ is non-negative), and $\deg(Q^{(p)}) = \deg(P) \leq d$.
Now apply
Corollary \ref{cor:descent2}  and
Corollary \ref{cor:main}.
\end{proof}

\section{Conclusions and Open Problems}\label{sec:conclusion}
In this paper we have proved asymptotically tight upper bounds on the
equivariant Betti numbers of symmetric real semi-algebraic sets. These bounds
are exponential in the degrees of the defining polynomials, and also in the
number of non-symmetric variables, but polynomial in the remaining parameters
(unlike bounds in the non-equivariant case which are exponential in the number
of variables). We list below several open questions and topics for future
research.

It would be interesting to extend the results in the current paper to
multi-symmetric semi-algebraic sets, where the symmetric group acts by
permuting blocks of variables with block sizes $>1$. As an immediate
application we will obtain extension of Theorem
\ref{thm:descent2-quantitative} 
to the
case where the projection is along more variables than one.

An interesting problem is to prove that the vanishing of the equivariant
cohomology groups in Theorems \ref{thm:main} occurs for dimension $\geq d$
(rather than $2d$).

Another direction (which has already being mentioned in Remark \ref{rem:multiplicities}) is to extend
the polynomial bounds obtained in this paper to multiplicities of other non-trivial irreducible representations of
$\mathfrak{S}_{\mathbf{k}}$ in the cohomology groups of symmetric real varieties or semi-algebraic sets
(viewed as an  $\mathfrak{S}_{\mathbf{k}}$-module), and to characterize those that could occur with positive multiplicities.
We address these questions in a subsequent paper.

In \cite{Church-et-al} the authors define a certain algebraic structure called $\mathrm{FI}$-modules.
For a finitely generated $\mathrm{FI}$-module $V$ over a field $\mathbb{F}$ of char $0$, for each $n\in \mathbb{Z}_{>0}$
there exists an $\mathbb{F}$-vector space $V_n$, the authors prove that the
dimension of $V_n$ is a polynomial in $n$ for all sufficiently large $n$ 
(see \cite{Church-et-al} for the necessary definitions). Amongst the primary examples of $\mathrm{FI}$-modules are certain sequences of 
$\mathfrak{S}_n$-representations, and as a consequence of the above result their dimensions can be expressed as a polynomial
in $n$. Our polynomial bounds on the $\mathfrak{S}_n$-equivariant Betti numbers of 
sequences of symmetric semi-algebraic sets 
(for example, consider the sequence of  real algebraic varieties defined by the sequence 
elementary symmetric polynomials $\left(e^{(n)}_d\right)_{n>0}$ of degree $d$ for some fixed $d$) 
suggest a connection with the theory of $\mathrm{FI}$-modules. It would be interesting to explore 
this possible connection.  

As mentioned in the Introduction, bounds on the 
ordinary (not equivariant)
Betti numbers
of semi-algebraic sets have found applications in theoretical computer
science, for instance in proving lower bounds for testing membership in
semi-algebraic sets in models such as algebraic computation trees. In this
context it would be interesting to investigate if the equivariant Betti
numbers can be used instead -- for example in proving lower bounds for
membership testing in symmetric semi-algebraic sets in an algebraic decision
tree model where the decision tree is restricted to use only symmetric
polynomials.

Finally, we have left open the problem of designing efficient (i.e.
polynomial time for fixed degree) algorithms for computing the individual
Betti numbers of symmetric varieties. In particular, we conjecture that for
every fixed $d$, there exists a polynomial time algorithm for computing the
individual Betti numbers (both ordinary and equivariant) of any symmetric
variety described by a real symmetric polynomial given as input.

\section*{Acknowledgments}{The authors thank A. Gabrielov for suggesting the
use of our equivariant bounds in the non-equivariant application described in
the paper. 
The authors gratefully acknowledge several comments and corrections from two anonymous referees
that greatly helped to  improve the paper.}

\bibliographystyle{abbrv}
\bibliography{master}

\end{document}